\documentclass{amsart}

\usepackage{amsmath}
\usepackage{amsthm}
\usepackage{amsfonts}
\usepackage{amsbsy}
\usepackage{amssymb}
\usepackage[initials]{amsrefs}
\newtheorem{theorem}{Theorem}
\newtheorem{proposition}{Proposition}

\newtheorem{lemma}{Lemma}

\newcommand{\R}{{\mathbb R}}
\newcommand{\Z}{{\mathbb Z}}
\newcommand{\C}{{\mathbb C}}

\newcommand{\set}[2]{ \left\{ #1 \ \left| \ #2 \right. \right\} }

\usepackage{multirow}
\usepackage{mathrsfs}

\newcommand{\lleq}{< \! \! <}
\newcommand{\ggeq}{> \! \! >}

\newcommand{\ind}{{\bf 1}}

\author{Yen Do}
\thanks{Y. Do was partially supported by NSF grant DMS-1201456}

\author{Philip T. Gressman}
\thanks{P. T. Gressman was partially supported by NSF grant DMS-1101393 and an Alfred P. Sloan Foundation Fellowship.}
\title[An operator van der Corput estimate]{An operator van der Corput estimate arising from oscillatory Riemann-Hilbert problems}

\begin{document}
\begin{abstract}
We study an operator analogue of the classical problem of finding the rate of decay of an oscillatory integral on the real line.   This particular problem arose in the analysis of oscillatory Riemann--Hilbert problems associated with partial differential equations in the Ablowitz-Kaup-Newell-Segur hierarchy, but is interesting in its own right as a question in harmonic analysis and oscillatory integrals.  As was the case in earlier work of the first author \cite{do2011}, the approach is general and purely real-variable.  The resulting estimates we achieve are strongly uniform as a function of the phase and can simultaneously accommodate phases with low regularity (as low as $C^{1,\alpha}$), local singularities, and essentially arbitrary sets of stationary points that degenerate to finite or infinite order.
\end{abstract}
\maketitle

\section{Introduction}
In this paper, we consider an operator analogue of the classical problem of finding the rate of decay for a oscillatory integral on the real line. This problem arose from \cite{varzugin1996} and previous work of the first author in \cite{do2011} where a real-variable approach akin to the classical stationary phase method was developed for oscillatory Riemann-Hilbert problems (RHPs). Oscillatory RHPs, in turn, play central roles in the study of long-time asymptotics for solutions of partial differential equations in the Ablowitz-Kaup-Newell-Segur (AKNS) hierarchy in the same way oscillatory integrals are important for linear PDEs, see e.g. \cite{bc1984} or \cite{dz1993}.  Before stating the specific problem and the main results, we include a brief historical discussion of the context where oscillatory RHPs appear.

It is a classical fact that solutions of the standard linear partial differential equations (PDEs) with constant  coefficients can be written as oscillatory integrals in the form
$$\int e^{it\phi(\xi)}f(\xi)d\xi.$$
Using this representation  long time asymptotics of these solutions can be obtained via the method of stationary phase. Here $f(\xi)$ is the Fourier transform of the corresponding initial data of the PDE, and $\phi$ is a polynomial that encodes the linear structure of the PDE and  the particular point $(x,t)$ where the long-time asymptotics is being investigated. In the classical approach, the first step in the analysis of oscillatory integrals is to use integration by parts to localize the integral to neighborhoods of stationary points $\{\phi'(\xi)=0\}$, and the second step is to reduce the phase to an appropriate  Taylor polynomial near each stationary point. Once the phase is polynomial, contour deformation can be used to extract the leading asymptotics; this final step is part of the method of steepest descent, which is an alternative to stationary phase for oscillatory integrals with analytic phases.  An alternate approach is provided by {van der Corput} estimates for oscillatory integrals, which, for each $\phi$ and $f$ with some regularity and decay, allow one to find the largest constant $\gamma>0$ such that
$$\left|\int e^{it\phi(\xi)}f(\xi)d\xi \right| = O(t^{-\gamma}).$$
In the end, the regularity and decay required of $f$ to deduce long-time asymptotics can be transferred to corresponding decay and regularity requirements for the initial data via the classical Riemann-Lebesgue correspondence principle. Note that van der Corput estimates immediately provide decay estimates for oscillatory integrals \emph{without} the need to perform the sequence of three steps mentioned above.  The reader not familiar with van der Corput estimates will find a thorough treatment in Stein \cite{steinsi}.

In the study of long-time asymptotics for a large class of one-dimensional nonlinear PDEs, one encounters nonlinear variants of oscillatory integrals. These are equations in the Ablowitz-Kaup-Newell-Segur (AKNS) hierarchy of completely integrable equations (which includes in particular NLS and mKdV); see e.g. \cites{bc1984,bdt1988}. Here, instead of $fe^{it\phi}$, one will have a matrix-valued function $G$ and some oscillatory entries, and, instead of an integral, one computes the residue(s) at $z=\infty$ of properly normalized multiplicative Riemann-Hilbert factors of $G$.\footnote{In a multiplicative Riemann-Hilbert factorization $G=G_-G_+$, we require $G_+$ to have analytic continuation to the upper half plane and $G_-$ to have analytic continuation to the lower half plane, and we require $G_\pm$ to have some normalization at infinity.} Note that if we instead used \emph{additive} Riemann-Hilbert components of $G$, the residues at $z=\infty$ will be some scalar multiple of the usual integral of $G$.\footnote{To see this, suppose $g$ is Schwartz. Modulo normalization the desired decomposition would be
$g= C_+g  - C_-g$
where $C_+$ and $C_-$ are Fourier multiplier operators with symbols $\ind_{\xi>0}$ and $-\ind_{\xi<0}$, or, equivalently, (up to some constant factors) the nontangential limits of the Cauchy transform
$$Cf(z) = \int_{-\infty}^\infty \frac {f(x)}{x-z}dx.$$
It is clear that the integral $\int g(x)dx$ shows up in the residues at $z=\infty$ of both components.} The task of finding long time asymptotics for solutions of these PDEs reduces to the task of analyzing long-time behavior for oscillatory Riemann-Hilbert problems. In this direction, various approaches have been developed: 
\begin{itemize}
\item In the seminal work \cite{dz1993}, P. Deift and X. Zhou introduced a nonlinear analogue of the  method of steepest descent and obtained long-time asymptotics for solutions of AKNS integrable equations (see also a previous work by Its \cite{its1981}).

\item Deift-Venakides-Zhou \cite{dvz1997} extended the method of Deift and Zhou to oscillatory RHPs with more delicate settings; see also \cite{dkmvz1999}. 

\item The methodology of Deift and Zhou requires the phase to be analytic.  K. McLaughlin and P. Miller \cites{mm2006,mm2008} extended Deift-Zhou's methodology to a method of $\overline{\partial}$-steepest descent, and consequently were able to handle phases with two Lipschitz derivatives (or, equivalently, three locally bounded derivatives).
\end{itemize}

The first work that uses the real-variable approach towards oscillatory Riemann-Hilbert problems is due to Varzugin \cite{varzugin1996} under the assumption that the phase $\phi$ has only primary stationary points. In \cite{do2011},  the first author further developed \cite{varzugin1996}, removing the above restriction on $\phi$ and, at the same time, obtaining some slight improvements over the two Lipschitz assumptions of \cites{mm2006, mm2008}. Assuming the phase has a finite number of stationary points, the argument obtained in \cite{do2011} is  analogous to the method of stationary phase for linear oscillatory integrals: 
\begin{itemize}
\item One starts with a localization argument (that is reminiscent of  integration by parts in linear theory) to reduce $G$ to a small neighborhood of the set of stationary points.  

\item After separating the local contributions  (to the nonlinear integral) of stationary points,  one approximates the phase by an appropriate polynomial near each stationary point. 

\item Once the phase is polynomial, one essentially appeals to the steepest descent methodology of Deift and Zhou to find the leading asymptotics of the nonlinear oscillatory integral. 
\end{itemize}

Recall that for classical oscillatory integrals the first two reductions are made possible thanks to integration-by-parts or van der Corput estimates. In the argument in \cite{do2011}, the first two steps use the following key estimate: if $C_-$ is the Fourier multiplier operator with symbol $\ind_{\xi<0}$ and if $\phi'\ge 0$ on the support of $f$, then
\begin{equation}\label{opvanderCorput}\|C_-(fe^{it\phi})\|_p = O(t^{-\gamma}), \ \ 2\le p\le \infty,
\end{equation}
where (given sufficient regularity of $f$ and $\phi$ and sufficient decay of $f$) the decay order $\gamma$ will depend on $p$ and the order of vanishing of $f$ at the stationary points $\{\phi'=0\}$.\footnote{By symmetry, an analogous statement holds for $C_+$, the Fourier multiplier operator with symbol $\ind_{\xi>0}$.} Furthermore, just as classical van der Corput estimates  immediately imply decay estimates for oscillatory integrals and, hence, solutions of linear PDEs  (without the need to go through the whole sequence of the stationary phase method), establishing \eqref{opvanderCorput} will also lead to decay estimates for AKNS nonlinear oscillatory integrals (and hence for solutions of nonlinear PDEs in the AKNS hierarchy) without the need to execute the full sequence of steps in the nonlinear stationary phase argument in \cite{do2011}. 

In the applications of \eqref{opvanderCorput} to the nonlinear stationary phase argument in \cite{do2011}, (modulo auxiliary factors) $f$ will be some nonlinear Fourier transform of the initial data of the underlying PDE. Using a nonlinear Riemann-Lebesgue transference principle of X. Zhou \cite{zhou1998}, one may transfer the decay and regularity required by \eqref{opvanderCorput} to requirements on the initial data of the equation. This is similar to the classical situation where the decay/regularity required by van der Corput estimates for oscillatory integrals could be transferred to conditions on the initial data.

On the applied side, \eqref{opvanderCorput} is related to a well-known phenomenon in the signal processing community. If we use the identity
$$C_- = \frac 1 2 (I- H)$$
where $I$ is the identity operator and $H$ is the Hilbert transform, it follows that the Hilbert transform of a bump function that oscillates with positive frequencies is similarly oscillatory. If the bump function oscillates with negative frequency only, we may use the $C_+$ analogue of \eqref{opvanderCorput}, which gives decay for $C_+$, to establish a similar fact. The fact that the Hilbert transform of a wavelet-like function (meaning a function with vanishing moments, decay, and regularity) is also wavelet-like is well-known 
throughout the signal processing community, see e.g. \cite{cu2011}, and one might view operator van der Corput estimates of the type \eqref{opvanderCorput} as more quantitative estimates for this phenomenon.

Our aim in this paper is to further investigate the operator van der Corput estimates \eqref{opvanderCorput} in the following direction: for each given phase $\phi$ and a given desired rate of decay $\gamma$ we would like to obtain sharp conditions on $f$ so that the estimate \eqref{opvanderCorput} holds. The versions of \eqref{opvanderCorput} obtained in \cite{do2011} require fairly stringent decay assumptions on $f$ and its  derivatives, and at the same time do not allow for fractional derivatives, which would be wasteful in potential applications to PDEs. These issues are addressed in the main results of this paper. Furthermore, the restriction $p\ge 2$ in \eqref{opvanderCorput} will be removed.

It is likely that the estimates in this paper will lead to improved estimates for the error terms in the long-time asymptotics for solutions of some AKNS equations with rough initial data;  however, this issue will not be explored in this paper. We plan to revisit this direction as well as other applications to oscillatory RHPs (such as in \cites{mm2006, mm2008}, see also \cites{ll2008,lubinsky2009} for  related work) in future work. We would like to point out that P. Deift and X. Zhou \cite{dz2003} have obtained essentially sharp estimates for the error terms in the long-time asymptotics of solutions of the NLS equation with initial data in a weighted Sobolev space; see also an important application of  \cite{dz2003} in \cite{dz2002}.

Throughout this paper,  the notation $A \lesssim B$ will mean that there is a finite constant $C$ such that $A \leq CB$ holds uniformly in the parameters of $A$ and $B$ (which will be explicitly noted when it is not clear).  Modifying the symbol $\lesssim$ with subscripts (i.e., $\lesssim_j$) indicates that the constant will depend on the value of the subscript.  Certain numbers and functions (e.g., $\kappa$ and $\alpha$ in what follows) will be explicitly identified as fixed, and so implicit constants will be allowed to depend on these fixed constants without further notice.  We define $A \gtrsim B$ and  $A \approx B$ similarly.  Finally, sentences of the form ``$A \lleq B$ implies $X$'' mean that there is a nonzero positive constant $c$, independent of the parameters of $A$ and $B$, such that whenever $A \leq c B$ is true for this constant $c$, ``$X$'' holds. (Typically such $c$ is a sufficiently small constant.)

We first state a particular case of our main results, formulated in the situation when the phase $\phi$ is polynomial. (This is the case for AKNS PDEs). Below $\mathcal L_p^\gamma$ denotes the fractional Sobolev seminorm
$$\|f\|_{\mathcal L_p^\gamma}:=  \|D^\gamma f\|_p$$
here $D^{\gamma}$ is the fractional derivative operator
$$\widehat{D^{\gamma}f}(\xi) = \widehat{f}(\xi) |\xi|^\gamma \ \ .$$

\begin{theorem}\label{t.polynomial} Let $\phi$ be a real polynomial of degree $k\ge 2$, and let $\alpha_1,\dots, \alpha_{k-1}$ be the (complex) zeros of $\phi'$. Then for any $1<p<\infty$ and any $\gamma>0$ it holds that
$$\|C_-(e^{it\phi}f\ind_{\phi'>0})\|_p \lesssim_{k,\gamma,p} t^{-\gamma} \Big[\Big\|\frac{f}{(\phi'\ell)^\gamma}\Big\|_{L^p} + \Big\|\frac{f}{(\phi')^\gamma}\Big\|_{\mathcal L_p^\gamma} \Big],$$
where $\ell$ is given by
$$\ell(x):=\min_{1\le j\le k-1} |x-\alpha_j| .$$
\end{theorem}

Theorem \ref{t.polynomial} is, 
in some sense, analogous to a weighted version of the sharp van der Corput's lemma for polynomials developed by Phong and Stein \cite{ps1992}. (Variations of this idea appear in work of Phong, Stein, and Sturm \cite{pss1999} and, more recently, in work of Wright \cite{wright2011}.)  The main difference is the phrasing of the estimate in terms of the distance function $\ell$ rather than in terms of root clusters. The approach used here happens to be convenient because it allows us to deduce Theorem~\ref{t.polynomial} as a corollary of more general theorems which can simultaneously accommodate phases with such features as low regularity (as low as $C^{1,\alpha}$), local singularities, or essentially arbitrary sets of stationary points that may degenerate to finite or infinite order.   Below we describe the setup for these results.

Let $\Omega$ be an open subset of $\R$ which could be unbounded. For a nonnegative integer $\kappa$ and  $\alpha \in [0,1]$ we will assume that the following hold:
\begin{itemize}
\item The phase $\phi$ is real valued on $\Omega$ and $\phi'>0$ on $\Omega$. 
\item The phase $\phi$ is also locally $C^{\kappa +1,\alpha}$ in $\Omega$, i.e., $\phi^{(\kappa+1)}$ exists and is locally H\"older continuous in $\Omega$ with exponent $\alpha$ (the implicit constant is not required to be uniform over $\Omega$).
\end{itemize} 
For such $\phi$ and $(\kappa,\alpha)$, we say that a function  $\ell$ defined on $\Omega$ is \emph{admissible} if: 
\begin{itemize}
\item The function $\ell$ vanishes on $\partial \Omega$ (a trivial constraint when $\Omega = \mathbb R$) and is nonnegative and $1$-Lipschitz, i.e.,
$$|\ell(x)-\ell(y)| \le |x-y| \qquad \forall \, x,y\in \Omega.$$
\item For any $x\in \Omega$ and any $1\le j\le \kappa$,
\begin{equation}\label{nphase1}
|\phi^{(j+1)}(x)| (\ell(x))^j \lesssim |\phi'(x)| \qquad \forall \, x\in \Omega.
\end{equation}
\item Uniformly over $x,y\in \Omega$, the inequality $|y-x| \lleq \ell(x)$ implies
\begin{equation}\label{nphase2}
\frac{|\phi^{(\kappa+1)}(x) - \phi^{(\kappa+1)}(y)|}{|x-y|^\alpha} (\ell(x))^{\kappa+\alpha} \lesssim |\phi'(x)|.
\end{equation}
\end{itemize}
An informal interpretation of \eqref{nphase1} and \eqref{nphase2} is that differentiation of $\phi'$ at $x$ is no worse than multiplication by $(\ell(x))^{-1}$. (Note that \eqref{nphase1} is not applicable if $\kappa=0$.) Alternatively, \eqref{nphase1} and \eqref{nphase2} can be viewed as pointwise upper bounds on $\ell$.  

We also introduce some designer spaces to measure regularity of functions on $\Omega$.  For any $p \in (1,\infty)$ and any $\beta \in (0,1)$, consider the following seminorm
\[ ||f||_{X^p_{\beta}(\Omega)} := \left( \int_{\Omega} \left| \sup_{0 < \delta < \ell(x)} \frac{1}{\delta^{1+\beta}} \int_{-\frac{\delta}{2}}^{\frac{\delta}{2}} |f(x + u) - f(x)| du \right|^p dx \right)^{\frac{1}{p}}. \]
When $\beta = 0$, we let $||f||_{X^p_0} := ||f||_{L^p}$, and when $\beta \geq 1$, we let $||f||_{X^p_\beta} := ||f^{(j)}||_{X^p_{\beta-j}}$, where $j$ is the largest integer less than or equal to $\beta$.
We note that the class of functions with finite $X_\beta^p$-seminorm includes the Bessel potential fractional Sobolev spaces ${{\mathscr L}^p_\beta}$, the Riesz potential Sobolev spaces ${\mathcal L}^p_{\gamma}$ already considered, and more exotic Lipschitz-type variants such as ${\Lambda_\beta^{p,p}}$ (for definitions of these spaces, see Stein \cite{steinsi}).  

We are now ready to state the main generalizations of Theorem~\ref{t.polynomial}:
\begin{theorem}\label{flashythm2}
Assume that  $\ell$ is admissible.  Fix $p \in (1,\infty)$  and let $\gamma:=\kappa + \alpha$.
Then for any $t > 0$  it holds that
\begin{equation} \label{wanted2}
||C_{-} ( e^{i t \phi} f\ind_\Omega ) ||_{L^p(\R)} \lesssim_{\kappa,\alpha} \frac{1}{t^{\gamma}} \left[  \left| \left| \frac{f}{(\phi' \ell)^\gamma} \right| \right|_{L^p(\Omega)} + \left| \left|   \frac{f}{(\phi')^\gamma}  \right| \right|_{X^p_\gamma (\Omega)}  \right] .
\end{equation}
\end{theorem}

\begin{theorem}\label{flashythm} Assume that $\ell$ is admissible.   Fix $p\in (1,\infty)$ and let $\gamma:=\kappa + \alpha$.
Then for any integer $k\ge \gamma$  and $t>0$ it holds that
\begin{equation} \label{wanted}
\|C_-(e^{it\phi}f\ind_{\Omega})\|_p \lesssim  \frac{1}{t^{\gamma}} \left[  \left| \left| \frac{f}{(\phi' \ell)^{\gamma}} \right| \right|_{L^p(\Omega)} + \left| \left|\frac{\ell^{k  - \frac{1}{p} + \frac{1}{q}}}{\ell^{\gamma}} \frac{d^k}{dx^k} \left( \frac{f}{(\phi')^\gamma}  \right) \right| \right|_{L^q(\Omega)}  \right] ,
\end{equation}
where  $q\in [1,p]$ is any index such that
$$\frac 1q \le \frac 1 p + \gamma - k\ \ ,$$
and the implicit constant may depend on $k$, $\kappa$, $\alpha$, $p$, and $q$.
\end{theorem}

When one considers the function $f$ to be fixed and $\phi$ is very smooth, Theorems \ref{flashythm2} and \ref{flashythm} can be interpreted as a guide for deducing the decay rate of $||C_{-} (e^{it \phi} f\ind_{\phi'>0})||_p$ given information about both the decay of $f$  towards the boundary and the smoothness of $f$.  Specifically, one chooses $\gamma$ as large as possible so that $(\phi' \ell)^{-\gamma} f$ belongs to $L^p$ of $\Omega$ and that the $k$-th derivative of $(\phi')^{-\gamma} f$ belongs to the appropriate weighted space.  The condition $(\phi' \ell)^{-\gamma} f \in L^p$ is necessary in the sense that the decay rate $t^{-\gamma}$ cannot necessarily be achieved if one tests against a less singular power of the weight $(\phi' \ell)$.  
When $\phi(x) = x^k$, Theorems \ref{flashythm2}  and \ref{flashythm} are also clearly sharp in the sense that both sides scale equally under the change $x \mapsto \lambda x$ for any $\lambda > 0$.  As may be correctly surmised from a comparison of these two theorems, there exist additional scale-invariant estimates which hold in this class of spaces $X^{p}_\beta$ which are a consequence of Sobolev-type embedding relationships.  We give some examples in Section \ref{s.appendix}, but for the most part, leave the details to the reader.

The practical question when {constructing} an admissible $\ell$ is the extent to which the $1$-Lipschitz condition on $\ell$ prohibits it from simply being set equal to the largest possible pointwise values allowed by \eqref{nphase1} and \eqref{nphase2} and the vanishing condition on the boundary. In the ``critical'' case $k-\gamma = \frac{1}{p} - \frac{1}{q}$ of inequality \eqref{wanted}, $\ell$ appears with a negative power in the first term of the right-hand side and is absent in the second term, so at least in this case it is beneficial to choose $\ell$ as large as possible.  
 In Section~\ref{s.Lipschitz}, we demonstrate several facts about admissible functions. We show, for example, that when $\phi'$ vanishes to finite order at some point on the boundary, $\ell$ can be taken to vanish to first order at that point. We also show that when $\phi'$ is a nonconstant polynomial, choosing $\ell(x)$ equal to the minimum distance from $x$ to the critical points of $\phi$ (throughout the complex plane) or to the boundary of $\Omega$ will be admissible with constants depending only on the degree of the polynomial (see Proposition~\ref{basicprop2}). Using this  fact, Theorem~\ref{t.polynomial} follows immediately from Theorem~\ref{flashythm2}.

Let us also briefly discuss the sharpness of the regularity conditions on $\phi$ in Theorems \ref{flashythm2} and \ref{flashythm}.  In the event that the derivative of the phase $\phi'$ is nonvanishing, one might imagine that decay as $t \rightarrow \infty$ in an estimate like \eqref{wanted} will automatically result.  While this may indeed be the case (and, in fact, will be the case if $\phi'$ has any H\"{o}lder continuity), it is easy to see that the decay rate at minimum cannot be uniform in $\phi'$ without extra regularity assumptions.
We can see this by looking at the following example.  Fix any positive $n \ge 2$ such that
\[ c_n := \int_0^1 e^{2 \pi i n x} e^{i \cos 2\pi x} dx  \neq 0 \]
(it is clear that such $n$ exists). If $\hat f$ is supported on $[-A,A]$, then for $t > 2A$, $( e^{2 \pi i n t \cdot + i \cos (2\pi t \cdot)} f)^\wedge(\xi) = c_n \hat f(\xi)$ on $[-A,A]$.  Thus, despite the fact that the derivative of the phase 
$$\phi_t(x) = 2 \pi n x + \frac{\cos(2\pi tx)-1}t$$ 
is comparable to $1$ on the entire real line, where both lower and upper bounds are uniform over $t$, we have
\[ ||C_{-} ( e^{i t\phi_t} f)||_2 \geq |c_n| ||C_{-} f||_2, \]
which does not go to zero for generic $f$.  Thus the assumptions \eqref{nphase1} and \eqref{nphase2} cannot be relaxed to require less regularity of $\phi'$ without losing uniformity of the result.

The rest of this paper is organized as follows: 
in Section \ref{s.Lipschitz}, we construct a fundamentally-important partition of unity associated to admissible functions as well as establish several useful facts about individual admissible functions and the class of admissible functions as a whole. Section \ref{s.proof} contains the proofs of Theorems \ref{flashythm2} and \ref{flashythm}, which is divided into two parts: integration-by-parts or stationary phase arguments (Section \ref{ss.ibp}) and the main decomposition and summation arguments (Section \ref{ss.decomp}).  We note that the details of our approach are substantially different from those appearing in \cite{do2011} or \cite{varzugin1996}. In particular, we use a Littlewood-Paley approach and do not appeal to the Hausdorff-Young inequality.  This allows us to remove the restriction $p\ge 2$ present in these previous works.  Finally, Section \ref{s.appendix} includes some examples of embedding relationships among the spaces $X^p_{\gamma}(\Omega)$  and various other well-known spaces.  In particular, we establish the desired relationship of ${\mathcal L}^{\gamma}_p$ and  $X^{p}_\gamma$ so that Theorem \ref{t.polynomial} follows from Theorem \ref{flashythm2}.  We also give two examples of Sobolev- and Poincar\'{e}-type embeddings among the $X^p_\gamma$ spaces themselves.

\section{Lipschitz function considerations}\label{s.Lipschitz}

\subsection{Property (L) and partitions of unity}
\label{partition}

In this section we establish a relationship between a generalization of the class of Lipschitz functions and certain very nice partitions of unity on the real line.  
We say that a function $h$ defined on the real line satisfies property (L) if
\begin{itemize}
\item The function $h$ is continuous and nonnegative.
\item  For any interval $J$, it holds that
\begin{equation} 
\mbox{either } \sup_{x \in J} h(x) \leq 2 \inf_{x \in J} h(x) \mbox{ or } \sup_{x \in J} h(x) < 2 |J|.  \label{scale} \end{equation}
\end{itemize}
If $h$ is defined on some interval which is not the entire real line, the natural definition of property (L) is to extend $h$ to all of $\R$ to be locally constant outside its original domain while choosing the relevant constant values to preserve continuity.  

The intuition behind \eqref{scale} is that property (L) holds when $h$ measures ``distance up to an order of magnitude.''  When $h(x)$ is large, this should be interpreted as asserting that $x$ is far away from whatever bad set the function $h$ implicitly identifies (namely, the set where $h=0$).  The exact value of $h$ does not hold deep significance, but when it is large at $x$, it should remain large on a proportionally large interval around $x$.  
The prototypical examples of functions with property (L) are $1$-Lipschitz functions, but, significantly, H\"{o}lder functions with sufficiently small norm can generate functions with property (L) as well.  In particular, if
\[ |h(x) - h(y)| \leq \alpha 2^{\alpha - 1} |x-y|^{\alpha} \]
uniformly for all $x, y \in \R$, then we may deduce that
\[ \sup_{x \in J} (h(x))^{\frac{1}{\alpha}} \leq \inf_{x \in J} (h(x))^{\frac{1}{\alpha}} + 2^{\alpha - 1} |J|^{\alpha} \left( \sup_{x \in J} (h(x))^{\frac{1}{\alpha}} \right)^{1 - \alpha}. \]
If $\sup_{x \in J} (h(x))^{\frac{1}{\alpha}} > 2 \inf_{x \in J} (h(x))^{\frac{1}{\alpha}}$ for any particular interval, then
\[ \frac{1}{2} \sup_{x \in J} (h(x))^{\frac{1}{\alpha}} < 2^{\alpha - 1} |J|^{\alpha} \left( \sup_{x \in J} (h(x))^{\frac{1}{\alpha}} \right)^{1 - \alpha}, \]
which implies that $h^{\frac{1}{\alpha}}$ has property (L).  It is also true that if $h$ has property (L) then $\epsilon h$ will also have property (L) for any $\epsilon \in [0,1)$.  We call the condition property (L) because any function satisfying \eqref{scale} is nearly Lipschitz, meaning specifically that there is a Lipschitz function whose pointwise ratio with $h$ is bounded above and below by universal constants, see Proposition~\ref{basicestprop1}.

The most significant feature of property (L) is that $h$ is close enough to a distance measure; so close, in fact, that it allows us to construct a very nice partition of unity on the set where $h$ does not vanish.  This partition of unity generalizes the Whitney decomposition (and essentially coincides with the Whitney decomposition when $h$ equals the distance to the complement of some open set).
\begin{lemma} \label{fancypart}
Suppose that $h$ has property (L).
Then there exists a collection $\mathcal G$ of closed intervals $I$ with the following properties:  
\begin{itemize}
\item The union of all intervals equals the set $U := \set{x \in \R}{h(x) > 0}$.  
\item Every interval $I \in {\mathcal G}$ intersects exactly two others: one on the left and one on the right.  Furthermore, the ratio of lengths of intersecting intervals is between $\frac{1}{2}$ and $2$.
\item Every pair of intervals $I, I' \in {\mathcal G}$ which do not intersect are separated by a distance of at least $\frac{1}{4} |I|$.
\item For any $I \in {\mathcal G}$, the double interval $2 I$  (having twice the length and the same center as $I$) lies strictly inside $U$.  
\item Every interval $I \in \mathcal G$ satisfies
\begin{equation}
 \max \left\{ |I|, \inf_{x \in I} h(x), \sup_{x \in I} h(x) \right\} \leq 6 \min \left\{ |I|, \inf_{x \in I} h(x), \sup_{x \in I} h(x) \right\}. \label{compare}
\end{equation}
\item There is a special partition of unity subordinate to $\mathcal G$: for each $I \in {\mathcal G}$, there is a nonnegative $C^\infty$ function $\psi_I$ supported on $I$ such that $|\psi^{(j)}_I(x)| \lesssim_j |I|^{-j}$ for each $j$ and
\[ \sum_{I \in {\mathcal G}} \psi_I (x) = \begin{cases} 1 & x \in U \\ 0 & x \not \in U \end{cases}. \]
\end{itemize}
\end{lemma}
\begin{proof}
We first restrict attention to the dyadic intervals---all intervals expressible as $[ j 2^{k}, (j+1) 2^k]$ for any $j, k \in \Z$.  This collection of intervals has several nice properties, chief among which is that any two dyadic intervals whose interiors intersect must have one contained in the other.  An almost equally important consequence of this fact is that each dyadic interval $J$ has a unique parent $J^+$ which contains $J$ and has twice the length.

Given the function $h$, we say that a dyadic interval $J$ is good when
\[ |J| \leq \inf_{x \in J} h(x) \mbox{ and } \inf_{x \in J^+} h(x) < |J^+|. \]
Since the good intervals are a subcollection of the dyadic intervals, they inherit the property that any two which overlap (i.e., have intersecting interiors) must be nested. But $\inf_{x \in J} h(x)$ decreases as $|J|$ increases, so clearly no good interval could be strictly contained in another.  Thus the good intervals are nonoverlapping.  

Next we observe that the good intervals cover $U$.  For any $x \in U$, there is a doubly-infinite nested sequence of dyadic intervals $\cdots \subset J_{0} \subset J_1 \subset \cdots$ whose intersection is $\{x\}$ and whose union is all of $\R$. If $2^k x$ is never an integer then this sequence is uniquely specified, but if $x = j 2^k$, then there are two possibilities: we can take $J_{k'} := [x,x+2^{k'}]$ for $k' \leq k$ and then take $J_{k'} := J_{k'-1}^+$ when $k' > k$, or we could instead have taken $J_{k'} := [x-2^{k'},x]$ for $k' \leq k$.
Since $h(x)>0$ and since $h$ is continuous, all sufficiently small dyadic intervals in this sequence will satisfy $|J| \leq \inf_{x \in J} h(x)$.  It is also clear that all intervals (in this sequence) of length greater than $h(x)$ will fail to satisfy the inequality.  Thus there will be a unique interval in the sequence which is good.  

Every good interval $J$ must have unique left and right neighbors $J_L$ and $J_R$ which intersect it exactly at its left and right endpoints, respectively. Indeed, if $J =[a,b]$ is good with length $2^k$, then we can extend the sequence $J_{k'} := [a-2^{k'},a]$ for $k \leq k$ to a doubly-infinite sequence and find some good $J'$ containing $a$.  In this case, $J' \neq J$ since the sequence we are choosing from explicitly does not contain $J$ as a possible choice. We may repeat this argument to find the right neighbor for $J$.

Under the assumption \eqref{scale}, for all good intervals $J$ we have
\begin{align}
\sup_{x \in J} h(x) & \leq 2 \inf_{x \in J} h(x), \label{compcond1} \\
\sup_{x \in J} h(x) & \leq \sup_{x \in J^+} h(x) < 4 |J|, \label{compcond2}
\end{align}
with strict inequality because we know that $\inf_{x \in J^+} h(x) < |J^+| = 2|J|$.
If $J$ and $J'$ share a common endpoint $a$ (i.e., they are neighbors), then
\[ |J'| \leq \inf_{x \in J'} h(x) \leq h(a) \leq \sup_{x \in J} h(x) < 4 |J|. \]
Since $|J'|$ and $|J|$ are both powers of $2$, it follows that $|J'| \leq 2|J|$.  Summing the geometric series gives that $h$ cannot vanish on the interior of $3 J$ for any good interval $J$.  

We now populate the collection ${\mathcal G}$ with intervals: we say that $I \in {\mathcal G}$ when $I = \frac{4}{3} J$ for some good $J$.  Since $\frac{4}{3} J$ covers at least $\frac{1}{12}$ and at most $\frac{1}{3}$ of $J$'s neighbors $J$, we will have that the distance between $I, I' \in {\mathcal G}$ is at least $\frac{1}{4} |I|$ when $I \cap I' = \emptyset$.  We know that
\begin{align*}
\frac{3}{4} |I| \leq \inf_{x \in I} h(x) & \leq \sup_{x \in I} h(x) \leq 2 \max \{ |I|, \inf_{x \in I} h(x) \}, \\
\inf_{x \in I} h(x) & \leq \inf_{x \in J} h(x) < 4 |J| = 3 |I|.
\end{align*}
These allow us to conclude \eqref{compare}.

We construct the partition of unity in the usual way:  fix some $C^\infty$ nonnegative function $\eta$ which is identically zero outside $[0,1]$ and identically one on $[\frac{1}{8}, \frac{7}{8}]$.  For any $I := [a,b] \in {\mathcal G}$, we define $\eta_I(x) := \eta( \frac{x-a}{b-a})$ and then let
\[ \psi_I(x) := \eta_I(x) \left( \eta_{I_L}(x) + \eta_I(x) + \eta_{I_R}(x) \right)^{-1}. \]
The denominator is uniformly bounded away from zero on $I$ since for any $I' \in {\mathcal G}$, $\eta_{I'}$ is identically one on the good dyadic interval from which it was derived. The comparability of lengths of adjacent intervals guarantees that $|\psi_I^{(j)}(x)| \lesssim_j |I|^{-j}$ for any $I$.  
\end{proof}

An almost immediate corollary is that any function $h$ with property (L) is comparable to a $C^\infty$ function on the set where $h \neq 0$.  The natural bounds on the various derivatives of this function and the comparability constants are universal.  We summarize this fact in the following proposition:
\begin{proposition}\label{basicestprop1}
Suppose that $h$ has property (L).  Then there exists a function $g$ such that $g$ is $C^\infty$ on the set where $h\ne 0$, and uniform over this set we have $h(x) \approx g(x)$  and $|g^{(j)}(x)| \lesssim_j (g(x))^{1-j}$.
\end{proposition}
The proof uses the partition of unity just constructed. Simply set
\[ g(x) := \sum_{ I \in \mathcal G} |I| \psi_I(x). \]
All the necessary estimates for $g$ and its derivatives follow immediately from the various estimates established in Lemma~\ref{fancypart}. We omit the details.

Next we give a pair of propositions which establish, among other things, that when $\ell$ is admissible with respect to $\phi$,  the supremum and infimum of $\phi'$ on an interval $I$ generated by $\ell$ (using Lemma~\ref{fancypart}) are comparable. 

\begin{proposition}\label{basicestprop2}
Suppose that $h$ has property (L). Let $U=\{x: h(x) > 0\}$ and let $\mathcal G$ be constructed by Lemma~\ref{fancypart}.
Let $g$ be such that at least one of the following holds:
\begin{itemize}
\item The function $g$ is $C^1$ and satisfies the following estimate on $U$:
\[ h(x) |g'(x)| \lesssim  |g(x)|.\]
\item For some $\alpha\in [0,1]$ the following holds: for any $x,y\in U$, $|x-y| \lleq h(x)$ implies
$$\frac{| g(x) - g(y) |}{|x-y|^\alpha} \lesssim \frac{|g(x)|}{h(x)^\alpha}.$$
\end{itemize}  
Then for all $I \in {\mathcal G}$,
\[ \sup_{x \in I} |g(x)| \lesssim \inf_{x \in I} |g(x)|. \]
\end{proposition}
\begin{proof}
Consider the first of the two alternatives.  Restrict attention to a particular $I \in {\mathcal G}$.  For points $x \in I$, we have  $|I| |g'(x)| \lesssim h(x) |g'(x)| \lesssim |g(x)|$, so that
\[ |g(x') - g(x)| \leq \int_x^{x'} |g'(u)| du \lesssim \frac{|x-x'|}{|I|} \sup_{y \in [x,x']} |g(y)| \]
when $x,x' \in I$.  Fix $N \ggeq 1$ to be chosen later. If we now subdivide $I$ into nonoverlapping intervals $I_1,\ldots,I_N$ of equal length, we may apply this inequality to conclude
\[ \sup_{x \in I_j} |g(x)| - \inf_{x \in I_j} |g(x)| \lesssim \frac{1}{N} \sup_{x \in I_j} |g(x)|. \]
Choosing $N \ggeq 1$ relative to the implicit constant will give that
\[ \sup_{x \in I_j} |g(x)| \leq 2 \inf_{x \in I_j} |g(x)| \]
for all $j=1,\ldots,N$.  Chaining these inequalities together, we conclude that the supremum on $I$ of $g$ does not exceed $2^N$ times the infimum.

For the second alternative, suppose that
$$\frac{|g(x)-g(y)|}{|x-y|^\alpha} \lesssim  \frac{|g(x)|}{h(x)^\alpha} $$
for any $x,y\in U$ such that $|x-y| \lleq h(x)$.
Fix an interval $I\in \mathcal G$ generated by $h$ and further divide $I$ into $N$ equal parts $I_1, \dots, I_N$ for some $N\ggeq 1$. We first choose $N$ larger than some absolute constant such that $|I_j| \lleq h(x)$ for any $x\in I$; this is always possible since $h(x) \approx |I|$ for every $x\in I$. It follows that for any $x,x'\in I$, we have
$$|g(x)-g(x')| \lesssim \frac{|x-y|^\alpha}{h(x)^\alpha}  |g(x)| \lesssim \frac 1 {N^\alpha} |g(x)|.$$
The rest of the argument is similar to the proof of part (1). 
\end{proof}

\begin{proposition}\label{basicestprop3}
Suppose that $h$ has property (L). Let $U=\{x: h(x) > 0\}$ and let $\mathcal G$ be constructed by Lemma~\ref{fancypart}.
Let $g$ be such that one of the following holds:
\begin{itemize}
\item The function $g$ is $C^1$ on $U$ and the following estimate holds uniformly on $U$:
\[ (h(x))^{\kappa+1} |g'(x)| \lesssim 1.\]
\item For some $\alpha \in (0,1]$ and $\kappa \ge 0$ the following holds: for any $x,y\in U$, $|x-y| \lleq h(x)$ implies
\[ \frac {|g(x) - g(y) |}{|x-y|^\alpha} \lesssim h(x)^{-\kappa - \alpha}. \]
\end{itemize}
Then for any interval $I \in {\mathcal G}$, either
\[ \sup_{x \in I} |g(x)| \lesssim \inf_{x \in I} |g(x)|\] 
or
\[\sup_{x \in I} |g(x)|  \lesssim |I|^{-\kappa}. \]
\end{proposition}
\begin{proof}
We may follow the same arguments used in the proof of Proposition~\ref{basicestprop2} to conclude that
\[ \sup_{x, y \in I} |g(x) - g(y)| \lesssim |I|^{-\kappa}. \]
Therefore we have that
$$\sup_{x\in I}|g(x)| - \inf_{x\in I} |g(x)| \lesssim |I|^{-\kappa}.$$
It follows that if $\sup_{x\in I}|g(x)| > 2\inf_{x\in I} |g(x)|$ then 
$$\frac 1 2 \sup_{x\in I}|g(x)| \le \sup_{x\in I}|g(x)| - \inf_{x\in I} |g(x)| \lesssim |I|^{-\kappa},$$
which is exactly the desired inequality.
\end{proof}

\subsection{Concerning the choice of an admissible $\ell$}

We briefly turn to the issue of admissibility. 
As mentioned in the introduction, conditions \eqref{nphase1} and \eqref{nphase2} can simply be regarded as upper bounds on the pointwise magnitude of $\ell$.  As already mentioned, in the ``critical cases'' of Theorems \ref{flashythm2} and \ref{flashythm}, one always benefits from choosing $\ell$ as large as possible.  Even in noncritical cases, there is still a penalty to be paid for choosing the pointwise values of $\ell$ to be too small.  The apparent difficulty of identifying an admissible $\ell$ is that the local constraints \eqref{nphase1} and \eqref{nphase2} are accompanied by a nonlocal constraint that $\ell$ be $1$-Lipschitz.  As it turns out, the constraint that $\ell$ be $1$-Lipschitz is not generally difficult to satisfy, and so admissible choices of $\ell$ exist which are not significantly smaller pointwise than is required by \eqref{nphase1} and \eqref{nphase2}.  In this section, we give two propositions which make this idea more precise. The first proposition makes several general observations about admissibility, and the second addresses the specific case of polynomial phases $\phi$.

Throughout this section, we fix $\kappa$ to be a nonnegative integer and $\alpha \in [0,1]$. We suppose that a real-valued phase $\phi$ is fixed on a domain $\Omega$ with $\phi'>0$ everywhere in $\Omega$ and that $\phi^{(\kappa+1)}$ is locally H\"older continuous with exponent $\alpha$.  

\begin{proposition} \label{basicprop1}
The following are true:  
\begin{itemize}
\item The pointwise maximum of any finite collection of admissible functions is also admissible.  If there is a uniform bound on the implied constants, then the supremum over an infinite collection of admissible functions will also be admissible.

\item There exists a function $\ell$ such that $\ell$ is positive on $\Omega$ and is admissible with respect to $\phi$, $\kappa$, $\alpha$ on $\Omega$.  Furthermore, this $\ell$ has $\ell(x)=0$ exactly on $\Omega^c$ (note that this is vacuous if $\Omega=\R$).

\item Suppose $\kappa \geq 1$ and that $h$ has property (L) and is nonvanishing on $\Omega$ and that
$x,y\in \Omega$, $|x-y|\lleq h(x)$, implies
\begin{equation}
 \left| \frac{\phi^{(\kappa+1)}(x)}{\phi'(x)} - \frac{\phi^{(\kappa+1)}(y)}{\phi'(y)} \right| \lesssim \frac{|x-y|^{\alpha}}{h(x)^{\kappa+\alpha}}. \label{showlip}
 \end{equation}
Then there is an admissible function $\ell$ such that
\[ \ell(x) \approx \min \left\{ \left| \frac{\phi'(x)}{\phi''(x)} \right|,\ldots, \left|  \frac{\phi'(x)}{\phi^{(\kappa+1)}(x)} \right|^{\frac{1}{\kappa}}, h(x) \right\} \]
for all $x \in \Omega$.
\end{itemize}
\end{proposition}
Remark: If $\phi$ is sufficiently nice we may take $h(x)=d(x,\Omega^c)$ in part (3) and construct $\ell$ that vanishes on the boundary of $\Omega$ with first order.
\begin{proof}
We approach the various points consecutively.  The first point is essentially a triviality.

[Proof of second point.] To begin, we identify a natural choice of $\ell$ satisfying \eqref{nphase1} and \eqref{nphase2}. Fix a positive $K$.  At any $x \in \R$, define 
$$\ell(x) :=\min \{ d(x, \partial \Omega), \ell_0(x)\},$$ 
where $d(x,\partial \Omega)$ is the distance to the boundary of $\Omega$ (which we take to be $\infty$ if $\Omega = \mathbb R$), and
\begin{align*}
\ell_0(x) := \mathop{\inf_{y, z \in \Omega}}_{1 \leq j \leq \kappa} \min \left\{  \vphantom{\frac{a^\frac{1}{p}}{b}} \right.    |x&-y|  + K  \left. \left( \frac{|\phi'(y)|}{ |\phi^{(j+1)}(y)|} \right)^{\frac{1}{j}}   \right., |x-y| + \\
   & + \left.  K \left( \frac{|y-z|^{\alpha} |\phi'(z)|}{ |\phi^{(\kappa+1)}(y) - \phi^{(\kappa+1)}(z)|}  \right)^{\frac{1}{\kappa+ \alpha}}  + |y-z|\right\}.
\end{align*}
This function is a pointwise infimum of $1$-Lipschitz functions, so it is also $1$-Lipschitz.  Note that $\ell$ as defined above is finite in all cases other than when $\phi'' \equiv 0$ and $\Omega = \R$; in that particular case, we may take $\ell$ to be any positive Lipschitz function on $\R$.  It is relatively easy to check \eqref{nphase1} and \eqref{nphase2}.  


[Proof of third point.] For each $j \leq \kappa$, let $r_j(x) :=  |\phi'(x) / \phi^{(j+1)}(x)|^{1/j}$.  Consider the partition of unity generated by $h$ using Lemma~\ref{fancypart}. 
Assuming \eqref{showlip} holds, Proposition~\ref{basicestprop3} guarantees that on every interval $I \in {\mathcal G}$ either $\sup_{x \in I} r_{\kappa} (x) \approx \inf_{x \in I} r_{\kappa}(x)$ or $\inf_{x\in I} r_{\kappa}(x) \gtrsim |I| \approx h(x)$ for $x\in I$.  In particular, then, 
$$\min \{ \inf_{y \in I} r_{\kappa}(y), |I| \} \approx \min \{ r_\kappa(x), h(x) \}$$
for all $x \in I$. Consequently, using our partition of unity and rescaling, we may construct a $1$-Lipschitz function $\tilde h$ on $\Omega$ with 
\[\tilde h(x) \approx \sum_{I \in {\mathcal G}} \min \{ \inf_{y \in I} r_{\kappa}(y), |I| \} \psi_I(x)  \approx \min \{ r_\kappa(x), h(x) \},  \]
and clearly $\tilde h(x)=0$ on $\Omega^c$.


Now let $K$ be large and consider the minimum
\[ \ell(x) := \frac 1 K \min\{ r_1(x), \ldots, r_{\kappa-1}(x), \tilde h(x) \}. \]
The claim is that this function is Lipschitz as well.  Differentiating,
\[ \frac{d}{dx} \left| \frac{\phi'(x)}{\phi^{(j+1)}(x)} \right|^{\frac{1}{j}} = \frac{1}{j} \left[ \frac{\phi''(x)}{\phi'(x)} - \frac{\phi^{(j+2)}(x)}{\phi^{(j+1)}(x)} \right] r_j(x), \]
so it follows that $|r_j'(x)| \leq \frac{1}{j} ( r_j(x) / r_1(x)) + \frac{1}{j}(r_j(x)/r_{j+1}(x))^{j+1}$.  For a fixed $x$, if the minimum in the definition of $\ell(x)$ equals $r_j(x)$ then it follows from the above estimate that $|r_j'(x)| \lesssim 1$ (since $r_j$ is dominated by $r_{j'}$ for $j' \neq j$).  Thus for some absolute constant $C$ the following holds: given any $x\in \Omega$, on a small neighborhood of $x$ we have $\ell(y) - \ell(x) \le \frac{C}K |y-x|$ (We do not include absolute values on the left-hand side of this inequality because the minimizing $j$ may differ at $x$ and $y$, but we nevertheless get an upper bound on $\ell(y)$.) From here, it is not hard to see that $\ell$ is $C/K$-Lipschitz in every connected component of $\Omega$. On the other hand, since $\tilde h$ vanishes outside $\Omega$, so does $\ell$. Therefore $\ell$ is $C/K$-Lipschitz on $\R$, and by choosing $K$ large we may assume that $\ell$ is $1$-Lipschitz.

By definition of $\ell$, it is clear that $\ell$ satisfies \eqref{nphase1} and $\ell$ vanishes on $\partial \Omega$. 
It only remains to show that $\ell$ satisfies \eqref{nphase2} when \eqref{showlip} is assumed to hold.  
When $|x-y|\lleq \ell(x)$ (with appropriate constant such that $|x-y|\lleq h(x)$ as required by \eqref{showlip}), we have
\begin{align*}
|\phi^{(\kappa+1)}(x) - &  \phi^{(\kappa+1)}(y)| \\ & \leq  \left| \frac{\phi^{(\kappa+1)}(x)}{\phi'(x)} - \frac{\phi^{(\kappa+1)}(y)}{\phi'(y)} \right| |\phi'(x)| + \left| \frac{\phi^{(\kappa+1)}(y)}{\phi'(y)} \right| |\phi'(x) - \phi'(y)|.
\end{align*}
Note that by ensuring that the implicit constant in $|x-y|\lleq \ell(x)$ is small we may assume without loss of generality that $|x-y| < |I|/4$ where $I$ is the interval in the collection $\mathcal G$ generated by $h$ that contains $x$. In particular we'll have $[x,y] \subset \Omega$. Hence using the Mean Value Theorem, $\kappa \ge 1$, and \eqref{nphase1}, we have
\begin{align*}
|\phi^{(\kappa+1)}(x) - &  \phi^{(\kappa+1)}(y)|  \lesssim \frac{|x-y|^{\alpha}|\phi'(x)|}{\ell(x)^{\kappa + \alpha}} + \frac{|\phi''(z)| |x-y|}{\ell(y)^\kappa} \\
& \lesssim \frac{|x-y|^{\alpha}|\phi'(x)|}{\ell(x)^{\kappa+\alpha}} + \frac{|\phi'(z)| |x-y|}{\ell(y)^\kappa \ell(z)} \ \ ,
\end{align*}
for some $z$ between $x$ and $y$. Assuming $|x-y|  < \frac 1 2 \ell(x)$, the fact that $\ell$ is $1$-Lipschitz guarantees that
$\ell(x)\approx \ell(y) \approx \ell(z),$ 
and using part (1) of Proposition~\ref{basicestprop2} we also have  
$\phi'(x) \approx \phi'(y) \approx \phi'(z).$ 
Consequently, we obtain \eqref{nphase2}.
\end{proof}


\begin{proposition}\label{basicprop2} Let $\phi$ be a polynomial of degree $d+1$ with $d\ge 1$.  Fix any nonempty $\Omega \subset \{\phi'>0\}$. Then the function $\ell$ which equals the minimum of the distances to the zeros of $\phi'$ and the boundary of $\Omega$ is admissible, and the implied constants depend only on the degree.
\end{proposition}

\begin{proof}

It suffices to assume $\phi'(z) := p(z) := \prod_{j=1}^d (z - z_j)$.  Let $\ell_0(z)$ denote the distance from $z \in \C$ to the nearest $z_i$ for $i=1,\ldots,d$.   Then if we let $S$ range over all subsets of $\{1,\ldots,d\}$ with cardinality $d-k$, we have
\[ p^{(k)}(z) = k! \sum_{\# S = d-k} \prod_{j \in S} (z - z_j). \]
But $|\prod_{j \in S} (z-z_j)| \leq (\ell(z))^{-k} |p(z)|$, so counting subsets gives
\[ \frac{|p^{(k)}(z)|}{|p(z)|} \leq \frac{d!}{(d-k)!}\ell_0(z)^{-k}. \]
Since $\ell_0$ is clearly $1$-Lipschitz on the real line, we have \eqref{nphase2} for all $j$ with constants depending only on the degree (since the constants can be taken to vanish for $j$ sufficiently large). Setting $\ell(x) := \min (\ell_0(x), d(x,\partial \Omega))$ yields a $1$-Lipschitz function which vanishes on the boundary of $\Omega$ and automatically satisfies the pointwise bounds required for admissibility.
\end{proof}

\section{Proofs of Theorems \ref{flashythm2} and \ref{flashythm}} \label{s.proof}

We come now to the proofs of the main theorems. 
We begin in Section \ref{ss.ibp} with a lemma which is a variation on the usual integration-by-parts argument. After this, we introduce the partition of unity from Lemma \ref{fancypart} to decompose the problem, estimate the various pieces, and conclude the proofs.

\subsection{Frequency decomposition of oscillating phases} \label{ss.ibp}
In this section we fix a closed interval $I\subset \R$ of finite length. As usual let $\kappa$ be a nonnegative integer and $\alpha \in [0,1]$. Our main lemma is as follows:
\begin{lemma}\label{ibplem}Assume $w \in C^{\kappa,\alpha}(\R)$ is supported in $I$ such that 
\begin{equation}\label{ihypw}
\left\{
\begin{aligned}|w^{(j)}(x)| & \lesssim  |I|^{-j},  ~ 0 \leq j \leq \kappa,  \\
|w^{(\kappa)}(x) - w^{(\kappa)}(y)| & \lesssim |I|^{-\kappa- \alpha} |x-y|^{\alpha} ,
\end{aligned} \right.
\end{equation}
for implicit constants independent of $x \in I$.  Assume also that $\phi \in C^{\kappa+1,\alpha}(I)$  is a real valued function, and let $\Lambda$ be a finite constant such that
\begin{equation}
 | \phi^{(j+1)}(x)| \lesssim \Lambda |I|^{-j}, ~ 1 \leq j \leq \kappa, \label{ihyp}
 \end{equation}
 (where we interpret this condition to be vacuous when $\kappa = 0$) and
 \begin{equation}
  |\phi^{(\kappa+1)}(x) - \phi^{(\kappa+1)}(y)| \lesssim \Lambda |I|^{-\kappa-\alpha} |x-y|^{\alpha}, \label{ihyp2}
\end{equation}
with implicit constants in both cases that are independent of $x, y \in I$ and $\Lambda$.
 Then there is a function $E_I$ such that,  for any positive $m$,
\begin{align}
 | e^{i \phi(x)} w(x) - E_I(x) | & \lesssim_m \frac{(\Lambda |I|)^{-\kappa-\alpha} }{(1 +   \Lambda d(x,I))^m} \label{approx}
 \end{align}
and $\widehat{E_I}(\xi) \equiv 0$ at all frequencies $\xi$ for which $\inf_{x \in I} |\xi - \frac{\phi'(x)}{2 \pi}| > \frac{\Lambda}{4 \pi}$.  The implicit constant in \eqref{approx} depends only on $m$ and the constants in \eqref{ihypw}--\eqref{ihyp2}.
\end{lemma}

\begin{proof}
We note that the conclusion is trivial when $\Lambda |I| \leq 1$ and also when $\kappa = \alpha = 0$ (in both cases, take $E_I \equiv 0$), so we will assume the contrary. 

To begin, we know that there is a constant $C$ derivable from either \eqref{ihyp} (when $\kappa \neq 0$) or \eqref{ihyp2} (when $\kappa = 0$) such that $|\phi'(x) - \phi'(x')| \leq C \Lambda$ for any two $x, x' \in I$: In the H\"{o}lder case (i.e., \eqref{ihyp2} when $\kappa=0$),  this is immediate from $|x-x'| \leq |I|$.  On the other hand, if \eqref{ihyp} holds with $j=1$, the asserted inequality follows from the Mean Value Theorem.   We note that both cases allow us to establish the slightly stronger inequality $|\phi'(x) - \phi'(x')| \lesssim \Lambda |I|^{-\alpha} |x-x'|^{\alpha}$, which we will use when it is convenient to do so.  Let
\begin{align*}
 \xi_0   := - \frac{1}{2}  \Lambda  +  \inf_{x \in I} \phi'(x), \ \
 \xi_1   := \frac{1}{2}  \Lambda  +  \sup_{x \in I} \phi'(x), \ \ 
 \tilde \Lambda  := \xi_1 - \xi_0.
 \end{align*}
 Because we already know how much $\phi'$ varies over $I$, we know that $ \Lambda \leq \tilde \Lambda \leq (C+ 1) \Lambda$.  Let $\psi$ be any fixed $C^\infty$ function which is identically one on $[\frac{1}{4(C+1)},1-\frac{1}{4(C+1)}]$ and identically zero outside $[0, 1]$.  Now define $E_I$ by
\[ \widehat{E_I}(\xi) := \widehat{e^{i \phi} w}(\xi) \psi \left( 2 \pi \frac{\xi - \xi_0}{\tilde \Lambda} \right). \]
(We use the normalization convention that
\[ \hat f(\xi) := \int e^{-2 \pi i x \xi} f(x) dx, \]
but note that this choice has little significance beyond introducing occasional factors of $\frac{1}{2 \pi}$ to various expressions.) 
Clearly every $\xi$ in the support of $\widehat{E_I}$ has distance at most $\frac{1}{4 \pi} \Lambda$ to $\frac{\phi'(x)}{2 \pi}$ for some $x \in I$.  If that distance happens to be less than $\frac{1}{8 \pi} \Lambda$, we know that $\widehat{E_I}$ will be identically one on a neighborhood of that $\xi$.  Finally, note that we may express $E_I$ as the convolution
\begin{equation} E_I(x) = \int_I e^{i \phi(y)} w(y) \chi(x-y) dy, \label{convof} \end{equation}
where $\hat \chi (\xi) := \psi(2\pi {\tilde \Lambda}^{-1}  (\xi - \xi_0))$.

First we establish \eqref{approx} for the tails.  Suppose $I = [a,b]$.  Integrating the derivative estimates \eqref{ihypw} for $w$ guarantee that $|w(x)| \lesssim  |x-a|^{\kappa+\alpha} |I|^{-\kappa-\alpha}$.  We combine this with the following estimate for the size of $|\chi(y)|$:
\[ |\chi(y)| \lesssim_m \frac{ \tilde \Lambda}{(1 + \tilde \Lambda |y|)^{m+\kappa+\alpha+2}}. \]
This estimate is easily proved by rescaling $\chi$ and noting that $\psi$ is a Schwartz function.  Consequently, the convolution formula \eqref{convof} gives
\[ |E_I(x)| \lesssim_m \int_{I}  \frac{ |x-a|^{\kappa+\alpha} |I|^{-\kappa-\alpha} \tilde \Lambda}{(1 + \tilde \Lambda |x-y|)^{m+\kappa+\alpha+ 2}} dy .\]
Assuming $x < a$, we can make a series of elementary estimates
\begin{align*}  \int_{I} &  \frac{ |x-a|^{\kappa+\alpha}  |I|^{-\kappa - \alpha} \tilde \Lambda}{(1 + \tilde \Lambda |x-y|)^{m+\kappa+\alpha+2}} dy  \leq \frac{|I|^{-\kappa-\alpha}}{(1 + \tilde \Lambda |x-a|)^m} \int \frac{|x-y|^{\kappa+\alpha} \tilde \Lambda}{(1 +  \tilde \Lambda |x-y|)^{\kappa+\alpha+2}} dy \\
& \leq  \frac{|I|^{-\kappa-\alpha}}{(1 + \tilde \Lambda |x-a|)^m} \int \frac{ \tilde \Lambda}{\tilde \Lambda^{\kappa+\alpha} (1 +  \tilde \Lambda |x-y|)^{2}} dy \approx \frac{(\Lambda |I|)^{-\kappa-\alpha}}{(1 +  \Lambda d(x,I))^m}.
\end{align*}
This clearly establishes \eqref{approx} when $x < a$.  The case $x > b$ is symmetric.

Now we consider the case when $x \in I$. We rewrite \eqref{convof} slightly:
\begin{align}
E_I(x) = \int e^{i (\phi(x + h)- h \phi'(x)) } w(x+h) e^{i  h \phi'(x) } \chi(-h) dh. \label{convo}
\end{align}
By construction of $\chi$, we have
\begin{equation} \int e^{i h \phi'(x)} \chi(-h) dh = \int e^{-i h \phi'(x) } \chi(h) dh = \hat \chi \left( \frac{\phi'(x)}{2 \pi} \right) = 1. \label{mean} \end{equation}
Thus we can say
\begin{equation}
\begin{split}
 E_I & (x) - e^{i \phi(x)} w(x) = \\ &  \int \left[ e^{ i (\phi(x + h)- h \phi'(x)) } w(x+h) - e^{i \phi(x)} w(x) \right] \left[ e^{ i  h \phi'(x) } \chi (-h) \right]  dh.
 \end{split} \label{itsadiff}
 \end{equation}
When $\kappa = 0$, we are ready to finish the lemma.
By a variant of the Mean Value Theorem,
\begin{equation} \left| e^{i (\phi(x +  h) - \phi(x) -  h \phi'(x))} - 1 \right| \leq  |h| \! \sup_{|\delta| \leq |h|} |\phi'(x+\delta) - \phi'(x)| \lesssim |h|^{1 + \alpha} \Lambda |I|^{-\alpha}, \label{mvtm} \end{equation}
so that
\begin{align*}
 \left| e^{i (\phi(x + h)- h \phi'(x)) } \right. &  w(x+h)   - \left. e^{ i \phi(x)} w(x) \right| \\
  \leq & ~ |w(x+h) - w(x)| + |w(x)|   \left| e^{i (\phi(x +  h) - \phi(x) -  h \phi'(x))} - 1 \right| \\
  \lesssim &~ |h|^{\alpha} |I|^{-\alpha} + |h|^{1 + \alpha} \Lambda |I|^{-\alpha}.
 \end{align*}
Thus by scaling of $\chi$, we have
\[ |E_I(x) - e^{i \phi(x)} w(x)| \lesssim_m \int \frac{|h|^{\alpha}}{|I|^{\alpha}} \left(1 + \Lambda |h|  \right) \frac{\widetilde {\Lambda}}{(1+|\widetilde \Lambda h|)^m}dh \lesssim (\Lambda |I|)^{-\alpha}. \]
This completes \eqref{approx} when $\kappa = 0$.

When $\kappa > 0$, we will need to integrate by parts.  Let
\[ \eta^{(-1)} (h) := - h e^{- i h \phi'(x)} \chi (h). \]
The Fourier transform of $\eta^{(-1)}$ equals, up to a fixed constant,
\[ \tilde \Lambda^{-1} \psi' \left( \tilde \Lambda^{-1} \left(\xi + \phi'(x) - \xi_0 \right) \right). \]
In particular, it is supported on an annulus $|\xi| \approx \Lambda$. This means that $\eta^{(-1)}$ is an indefinitely integrable Schwartz function, meaning that it equals the $k$-th derivative of some Schwartz function $\eta^{(-k-1)}$ for any $k$.  By scaling, these Schwartz functions satisfy the estimates
\begin{equation} |\eta^{(-k)} (h)| \lesssim_{k,m} \frac{ \tilde \Lambda^{1-k}}{(1 + |\tilde \Lambda h|)^m}, \label{sdec} \end{equation}
including the case $k=1$ (uniformly in the choice of $x \in I$).

The basic technique we will use to establish \eqref{approx} on $I$ is integration-by-parts, but in a somewhat nonstandard form.  To begin, we develop some simplifying notation.  Let $\beta := (\beta_0, \beta_1,\ldots)$ be any sequence of nonnegative integers which equals zero for all but finitely indices.  We will define $\# \beta := \beta_0 + \sum_j j \beta_j$ for any such $\beta$, as well as a term $T^{\beta}(y,h)$:
 \[ T^{\beta}(y,h) \! := \eta^{(-\# \beta)}(-h) e^{ i (\phi(y) - (y-x) \phi'(x))} w^{(\beta_0)} (y) ( \phi'(y) - \phi'(x))^{\beta_1} \! \prod_{j \geq 2} (\phi^{(j)}(y))^{\beta_j}, \]
 where we understand $T^{\beta}(y,h) = 0$ when $y \not \in I$ because $w$ is supported on $I$.  We also take the understanding that the product ranges only over those indices $j$ for which $\beta_j \neq 0$.  We return to the expression \eqref{itsadiff} and write
 \begin{align*}
  e^{ i (\phi(x + h)- h \phi'(x)) } & w(x+h) - e^{i \phi(x)} w(x) \\
    = ~ &  h \int_0^1 i ( \phi'(x+\theta h) - \phi'(x)) e^{ i (\phi(x + \theta h)- \theta h \phi'(x)) } w(x+ \theta h) d \theta \\
    &  + h \int_0^1 e^{ i (\phi(x + \theta h)- \theta h \phi'(x)) } w'(x + \theta h) d \theta,
 \end{align*}
 which, when substituted back into \eqref{itsadiff} gives
 \begin{align}
E_I & (x) - e^{i \phi(x)} w(x)
   =   \int \int_0^1  \left[  T^{(1,0)}(x+\theta h,h) \right.  \left. + i T^{(0,1)}(x+\theta h, h) \right] d \theta dh, \label{fulldiff}
\end{align}
where $(1,0) = (1,0,0,0,\ldots)$ and $(0,1) = (0,1,0,0,0,\ldots)$.  If we let $\epsilon_j$ be the sequence which equals $-1$ at $j$, $+1$ at $j+1$, and zero everywhere else, then integrating by parts (integrating $\eta^{(-\#\beta)}(-h)$ with respect to $h$ and differentiating everything else with respect to $h$) gives that
\begin{equation}
\begin{aligned}
\int T^{\beta} & (x+\theta h,h) d h \\ & = \theta \int \left[  T^{\beta + (1,0)} + i T^{\beta + (0,1)} + \sum_{j \geq 1} \beta_j T^{\beta + \epsilon_j} \right] (x + \theta h,h) dh. 
\end{aligned} \label{ipbfact}
\end{equation}
(In particular, note that $T^{\beta + \epsilon_j}$ appears in the sum with nonzero coefficient only when $\beta + \epsilon_j$ has all nonnegative values.)
Since we know that \eqref{fulldiff} holds,
and we know that integrating $T^{\beta}$ by parts introduces new terms $T^{\beta'}$ with $\# \beta' = \# \beta + 1$,  we may repeatedly and selectively integrate individual terms on the right-hand-side of \eqref{fulldiff} 
until we arrive at an equality of the form 
\begin{equation}  E_I(x) - e^{i \phi(x)} w(x) = \int_0^1 \int \sum_{\beta \in B} C_{\beta,\kappa} \theta^{\#\beta - 1} T^{\beta} (x + \theta h,h)  dh d \theta \label{bind} \end{equation}
 with the sum over $\beta$ ranging only over indices $B$ which have at least one of the three following properties:
 \begin{itemize}
 \item The element $\beta$ equals $(\kappa,0,0,\ldots)$
 \item The element $\beta$ is zero everywhere except position $\kappa+1$, where it equals $1$.
 \item The element $\beta$ satisfies $\beta_0 + \beta_1 + \sum_{j \geq 2} (j-1) \beta_j \geq \kappa + \alpha$.
 \end{itemize}
 We prove this by induction. 
 Let $B_{1} := \{ (1,0), (0,1) \}$.  By \eqref{fulldiff}, we have \eqref{bind} with $B$ replaced by $B_1$.  We inductively build index sets $B_2,\ldots,$ so that the corresponding analogue of \eqref{fulldiff} holds for each $B_j$.  The rule is as follows:  any index in $B_j$ will belong to $B_{j+1}$ when it satisfies any of the three properties above (this corresponds to not integrating this term by parts).  Any term in $B_j$ which does not satisfy any of the three properties will be replaced in $B_{j+1}$ by the corresponding union of terms from the right-hand side of \eqref{ipbfact}.  One might worry that terms involving $w^{(\kappa)}$ or $\phi^{(\kappa+1)}$ cannot be integrated by parts since we do not assume that subsequent derivatives exist. This is not a problem, since any such term must automatically satisfy at least one of the three properties: if either $\beta_0 > \kappa - 1$ or $\beta_{\kappa+1} \neq 0$, then the only potential cases in which the third property fails are {\it exactly} the first two cases.  Thus, as soon as a term arises involving $w^{(\kappa)}$ or $\phi^{(\kappa+1)}$, it will satisfy one of the three properties and thus will not be integrated by parts any further.  We also note that the process stabilizes after $2\kappa+2$ steps, i.e., $B_{j+1} = B_{j}$ for any $j \geq \kappa+1$.  This is because when $T^{\beta}$ is integrated by parts, it is replaced by terms $T^{\beta'}$ which have $\# \beta' = \# \beta + 1$.  Since $\beta_0 + \beta_1 + \sum_{j \geq 2} (j-1) \beta_j \geq \frac{1}{2} \# \beta$, any term which is integrated by parts $2 \kappa+2$ times for failing to satisfy any of the three properties will then satisfy the third property.

When the third property is satisfied, we can estimate $T^{\beta}(x+\theta h, h)$ (assuming without loss of generality that $x + \theta h \in I$) using \eqref{ihypw}, \eqref{ihyp}, and \eqref{sdec}:
\begin{align*} | T^{\beta}(x+\theta h,h)| & \lesssim_m \frac{\tilde \Lambda^{1-\# \beta}}{(1 + |\tilde \Lambda h|)^m} |I|^{-\beta_0} ( |\theta h| \Lambda |I|^{-1} )^{\beta_1} \prod_{j \geq 2} \left( \Lambda |I|^{-j+1} \right)^{\beta_j}, \\
& \lesssim_m \frac{\tilde \Lambda^{1+ \beta_1} |\theta h|^{ \beta_1}}{(1 + |\tilde \Lambda h|)^m} (\Lambda |I|)^{-\beta_0} ( \Lambda |I|)^{- \beta_1} \prod_{j \geq 2} ( \Lambda |I|)^{-(j-1)\beta_j}.
\end{align*}
(We have implicitly used the fact that $\kappa \neq 0$ to conclude that $|\phi'(x) - \phi'(x')| \lesssim \Lambda |I|^{-1}  |x-x'|$ on $I$.) 
We conclude that
\[ \int \int_0^1 |T^{\beta}(x+\theta h,h)| d \theta dh \lesssim  (|I| \Lambda)^{- \kappa-\alpha }. \]
Note that the total number of $\beta$s satisfying the third condition above is finite depending only on $\kappa$, so summing over all such $\beta$ gives an estimate at least as good as \eqref{approx}.

Thus the two specific terms corresponding to the first two acceptable properties of $\beta \in B$ are the only remaining terms to consider, and in both cases we will use the H\"{o}lder condition on the highest derivatives of $w$ and $\phi$, respectively. In the first case, just as in \eqref{mvtm}, we have
\begin{align*}  \left| e^{i (\phi(x + \theta h) - \theta h \phi'(x))} w^{(\kappa)}(x + \theta h) \right. & - \left. e^{i \phi(x)} w^{(\kappa)}(x) \right| 
\\ \leq |w^{(\kappa)}(x + \theta h) - w^{(\kappa)}(x)| & + |w^{(\kappa)}(x)| \left| e^{ i (\phi(x + \theta h) - \phi(x) - \theta h \phi'(x))} - 1 \right| \\
& \lesssim (\theta |h|)^{\alpha} |I|^{-\kappa - \alpha} + (\theta |h|)^{1 + \alpha} \Lambda |I|^{-\kappa - \alpha}.
\end{align*}
Now, since $\int \eta^{(-\kappa)}(-h) dh = 0$, this inequality immediately implies that
\begin{align*}
 \left| \int_0^1 \int T^{(\kappa,0,\ldots)}(x + \theta h, h) d h d \theta \right| &\\
  \leq \int_0^1 \int  |T^{(\kappa,0,\ldots)} & (x+\theta h, h) - T^{(\kappa,0,\ldots)}(x,h) | dh d \theta \\ \lesssim \int |\eta^{(-\kappa)}(-h)| & ( |h|^{\alpha} |I|^{-\kappa - \alpha} + |h|^{1 + \alpha} \Lambda |I|^{-\kappa - \alpha}) dh \\
 & \lesssim (\Lambda |I|)^{-\kappa - \alpha}.
 \end{align*}
 In the second case, we use
 \begin{align*}  \left| e^{i (\phi(x + \theta h) - \theta h \phi'(x))} \phi^{(\kappa+1)}(x + \theta h) \right. & - \left. e^{i \phi(x)} \phi^{(\kappa+1)}(x) \right| \\ \lesssim  & \Lambda (\theta |h|)^{\alpha} |I|^{-\kappa - \alpha} + (\theta |h|)^{1 + \alpha} \Lambda^2 |I|^{-\kappa - \alpha}
\end{align*}
 instead. The difference is an extra factor of $\Lambda$, but this is offset by the fact that $\# \beta = \kappa+1$ for this term, so integrating against $\eta^{(-\kappa-1)}$ instead of $\eta^{(-\kappa)}$ exactly cancels this extra $\Lambda$.  Thus the proof is complete.
\end{proof}

\subsection{Main decomposition} \label{ss.decomp}
In this final subsection, we invoke Lemmas \ref{fancypart} and \ref{ibplem} to prove Theorems \ref{flashythm2} and \ref{flashythm}. 
We first make several reductions. We first observe that the case $\gamma = 0$ is trivial from the boundedness of $C_{-}$ on $L^p$ for $1 < p < \infty$, so we specifically assume $\gamma > 0$.  We will prove the theorems under the assumption $t=1$ since all hypotheses are invariant under positive rescalings of $\phi$.
Restricting $\Omega$ as necessary, we may assume that the chosen admissible $\ell$ does not vanish on $\Omega$, since if it did, any $f$ for which the right-hand side \eqref{wanted} could be interpreted as finite would have to vanish almost everywhere on the set where $\ell$ vanishes.  It must also be the case that $(\phi' \ell)^{-\gamma} \in L^p(\Omega)$, and since $\phi' \ell$ is nonzero without loss of generality and approximately constant on the intervals $I \in {\mathcal G}$ generated by $\ell$, we must have that $f$ is locally in $L^p$ in $\Omega$.

Given an admissible $\ell$, we generate a partition of unity as provided by Lemma~\ref{fancypart}, and we restrict our attention to those elements ${\mathcal G}_+ \subset {\mathcal G}$  contained in $\Omega$  (elements of the partition must be either contained in $\Omega$ or its complement since $\ell$ is nonvanishing on every interval $I \in {\mathcal G}$ but vanishing on $\partial \Omega$).  By Proposition~\ref{basicestprop2} and \eqref{nphase1}--\eqref{nphase2}, $\phi'$ and  $\ell$ are both approximately constant on every $I \in {\mathcal G}$.  For each interval, we define $\Lambda_I := \inf_{x \in I} \phi'(x)$. We will further divide $\mathcal G_+$ as follows: let the collection of intervals in ${\mathcal G}_+$ with $\Lambda |I| \geq 1$ be called ${\mathcal G}_L$ and let the rest of $\mathcal G_+$ belong to the collection ${\mathcal G}_T$ (``L'' for the fact that the product is large, and ``T'' for the fact that $E_I$ is trivial). On each interval $I$, we construct a band-limited function $E_I$ meant to approximate $(\phi')^{\gamma} e^{i \phi}$:
\begin{itemize}
\item When $I \in {\mathcal G}_L$, we will apply Lemma \ref{ibplem} on the interval $I$ to construct a function $E_I$ with Fourier support contained in $[\frac{1}{4 \pi} \Lambda_I, \infty)$ such that
\begin{equation} 
\left| e^{i \phi(x)} (\phi'(x))^{\gamma} \psi_I(x) - E_I(x) \right| \lesssim_m \frac{|I|^{-\gamma}}{(1 + \Lambda_I d(x, I))^m}. \label{lemout} \end{equation}
(Specifically, apply Lemma \ref{ibplem} with $w := |\Lambda_I|^{-\gamma} (\phi'(x))^{\gamma} \psi_I$ and rescale by $|\Lambda_I|^{\gamma}$ {\it ex post}.)  
\item When $I \in {\mathcal G}_T$, we take $E_I \equiv 0$. 
\end{itemize}

Next, without loss of generality we may assume that $f$ vanishes outside $\Omega$. Define
\[ F(x) := (\phi'(x))^{-\gamma} f(x) \ind_{\Omega}(x). \]
On each interval $I \in {\mathcal G}_L$, we will write $F$ as a sum $F = F^I + F_I$,where $F_I$ has Fourier transform supported in $[-\frac{1}{8 \pi} \Lambda_I, \frac{1}{8 \pi} \Lambda_I]$ (the ``low frequency part'') and $F^I$ is the ``high frequency'' complement.  The details of this decomposition will be explained shortly.
Using this notation, we expand
\begin{align}
 e^{i \phi} \ind_{\Omega} f & = \sum_{I \in {\mathcal G}} e^{i \phi} (\phi')^{\gamma} \psi_I F  =: S_1 + S_2 + S_3 + S_4 \label{decomp}
 \end{align}
 with the sums $S_1$--$S_4$ being given formally by
  \begin{align*}
  S_1  & := \sum_{I \in {\mathcal G}_T} e^{i \phi} (\phi')^{\gamma} \psi_I F,  & S_2 & := \sum_{I \in {\mathcal G}_L} \left[ e^{i \phi} (\phi')^{\gamma} \psi_I - E_I \right] F_I, \\
  S_3 & := \sum_{I \in {\mathcal G}_L} e^{i \phi} (\phi')^{\gamma} \psi_I F^I,  & S_4 & := \sum_{I \in {\mathcal G}_L} E_I F_I.
  \end{align*}
The sum $S_1$ is trivial: $|S_1| \lesssim |f| \ind_{|\phi'| \ell \lesssim 1}$, so trivially
\[ ||S_1||_{L^p(\R)} \lesssim \left| \left| \frac{f}{(\phi' \ell)^{\gamma}} \right| \right|_{L^p(\Omega)} \]
uniformly in $f$, $p$, etc.  The sum $S_4$ is also trivial since $E_I F_I$ has Fourier support on the positive half-line.  If we assume that $f \in L^p$ and prove the the sums $S_2$ and $S_3$ belong to $L^p$ as well, then $S_4$ will be an $L^p$ function with Fourier support on the positive half-line, so we will simply have
\[ C_{-} S_4 = 0. \]
The difficult terms are $S_2$ and $S_3$.  We consider $S_2$ first.  Let us assume that $F_I$ has been chosen so that
\begin{equation} |F_I(x)| \lesssim_m \int \frac{\Lambda_I |f(y)| dy}{(1 + \Lambda_I |x-y|)^m}. \label{kernelest} \end{equation}
By virtue of the triangle inequality, we have that
\[ 1 + \Lambda_I d(y,I) \leq (1 + \Lambda_I |x-y|) (1 + \Lambda_I d(x,I)), \]
so there is a corresponding integral inequality:
\begin{align*}
\frac{1}{(1 + \Lambda_I d(x,I))^{m}}  \! \int  \! \! \frac{\Lambda_I |f(y)| dy}{(1 + \Lambda_I |x-y|)^{2m}} \lesssim_m  \int \! \! \frac{\Lambda_I |f(y)| dy }{(1 + \Lambda_I |x-y|)^m (1 + \Lambda_I d(y,I))^m}.
\end{align*}
We will combine this inequality with \eqref{lemout} and the fact that
\[ \int |h(x)| \int \frac{\Lambda_I |g(y)| dy}{(1 + \Lambda_I |x-y|)^m}dx \lesssim_m \int |g(x)| (Mh)(x)dx, \]
where $M$ is the Hardy-Littlewood maximal operator (which is bounded on $L^p$ for $p > 1$). By duality, we come to the conclusion that
\[ ||S_2||_{L^p(\R)} \lesssim_m \left| \left| \sum_{I \in {\mathcal G}_L} \frac{|f|  |I|^{-\gamma}}{(1 + \Lambda_I d(\cdot, I))^m} \right| \right|_{L^p(\Omega)} \lesssim \left| \left| \frac{f}{(\phi' \ell)^{\gamma} } \right| \right|_{L^p(\Omega)} \]
once we establish the following estimate:
\begin{lemma}  Suppose $b < 0$.   Then
\begin{equation}
\sum_{I \in {\mathcal G}_L}  \frac{|I|^{b}}{(1 + \Lambda_I d(x,I))^m} \lesssim  (\ell(x))^b. \label{sumelmma}
\end{equation}
for all $m \geq 1 - b$. 
\end{lemma}
\begin{proof}
For each $I \in {\mathcal G}$, let $I_*$ be the union of $I$ and its left and right neighbors.  For any $m \geq 1 - b$,
\begin{align*}
 \frac{|I|^{b}}{(1 + \Lambda_I d(x,I))^m} &  \leq   |I|^{b} \ind_{I^*}(x) + \frac{ \Lambda_I^{b-1} |I|^{b} \ind_{(I^*)^c}(x)}{(d(x,I))^{1-b}} \leq |I|^{b} \ind_{I^*}(x) + \frac{  |I| \ind_{(I^*)^c}(x)}{(d(x,I))^{1-b}}.
 \end{align*}
Because the intervals $I^*$ have finite overlap, the sum over I of $|I|^{b} \ind_{I^*}(x)$ will be bounded by $(\ell(x))^b$.   Furthermore, because $x \not \in I^*$, we know $d(x,I) \gtrsim \ell(x)$; consequently
\[  \frac{ |I| \ind_{(I^*)^c}(x)}{(d(x,I))^{1-b}} \lesssim \int_{ I } \frac{\ind_{|y-x| \gtrsim \ell(x)}}{|x-y|^{1-b}} dy. \]
Summing over $I$ is now trivial, since
\[ \int \frac{\ind_{|y-x| \gtrsim \ell(x)}}{|x-y|^{1-b}} dy \lesssim (\ell(x))^{b}, \]
which is exactly as desired.
\end{proof}

The final piece of the decomposition \eqref{decomp} is the sum $S_3$.  We will fix a large integer $N$ and a $C^\infty$ function $\Psi$ with rapid decay and integral one whose Fourier transform is supported in $[-\frac{1}{8 \pi},\frac{1}{8 \pi}]$.   For any $h \in \R$, let $D_\delta F(x) := F(x+\delta) - F(x)$.  On each interval $I \in {\mathcal G}_L$, we define
\[ T_{I}^0 F(x) := \int  \Lambda_I \Psi( \Lambda_I u) D_u^N F(x) du. \]
This operator $T_I^0$ is essentially a convolution operator:
\[ T_I^0 F(x) = (-1)^{N} F(x) + \sum_{j=1}^N (-1)^{N-j} \binom{N}{j} \int F(x+j u) \Lambda_I \Psi(\Lambda_Iu) du. \]
We will define the splitting $F = F_I + F^I$ by taking $F_I := F + (-1)^{N+1} T_I^0 F$.  By design, $F_I$ will have Fourier support in $[-\frac{1}{8 \pi} \Lambda_I, \frac{1}{8 \pi} \Lambda_I]$ and will satisfy \eqref{kernelest}.

Now define $\Psi_I(x) := \Lambda_I \Psi(\Lambda_I x) \ind_{|x| \leq c |I|}$ and  \begin{equation} T_{I} F(x) := \int  \Psi_I(u) D_u^N F(x) du. \label{tidef} \end{equation}
The constant $c$ will be chosen sufficiently small so that $Nc|I| < \frac{1}{2} \ell(x)$ whenever $x \in I$.
Because $| \Lambda_I \Psi(\Lambda_I x) - \Psi_I(x)| \lesssim_m (\Lambda_I |I|)^{-m} \Lambda_I (1 + \Lambda_I |x|)^{-m}$, we can follow the derivation already carried out beginning with \eqref{kernelest} to conclude
\[ \left| \left| \sum_{I \in {\mathcal G}_L} e^{i \phi} (\phi')^{\gamma}  \psi_I \left[ T_{I} F - T_{I}^0 F \right] \right| \right|_{L^p(\R)} \lesssim \left| \left| \frac{f}{(\phi' \ell)^{\gamma}} \right| \right|_{L^p(\Omega)}. \]
Thus, it suffices to estimate the $L^p$ norm of the modified term
\[ \tilde S_3 := \sum_{I \in {\mathcal G}_L} e^{i \phi} (\phi')^{\gamma} \psi_I T_I F. \]
The advantage of doing so is that, when $c$ is chosen small enough (depending on Lemma \ref{fancypart}) and when $x \in I$, $T_I F$ will only depend on the pointwise values of $F$ in either $I$ or its left and right neighbors (i.e., the interval previously identified as $I^*$). 

The first feature to observe is that there is an integration by parts which can occur inside $T_I F$.  Specifically, if $F$ has a locally integrable weak derivative on $\Omega$ (by which we mean that the test functions will be compactly supported strictly away from $\partial \Omega$), we will have
\[ D_h F(x) := h \int_{0}^1 F'(x + \theta h) d \theta \]
at every $x \in \Omega$ for $h$ small enough that $[x,x+h] \subset \Omega$.  Plugging into the definition of $T_I$ and changing variables gives that
\begin{equation} T_{I} F (x) = \int_{[0,1]^k} \int \Psi_I(u) u^k D_u^{N-k} F^{(k)}(x+u(\theta_1+\cdots+\theta_k)) du d \theta_1 \cdots d \theta_k.  \label{tiop} \end{equation}
If we choose $N = k$, we will have that $T_I F$ is simply a convolution of $F^{(k)}$ with a kernel $\Psi_I^{(k)}$ given by
\[ \Psi^{(k)}_I(u) := \int_{[0,1]^k} \Psi_I\left(\frac{u}{\theta_1+\cdots+\theta_u} \right) u^{k} \frac{d \theta_1\cdots d\theta_k}{(\theta_1 + \cdots + \theta_k)^{k+1}}. \]
The transformation $\Psi_I \mapsto \Psi_I^{(k)}$ looks mysterious, but since $\theta_i \in [0,1]$, we have that the support of $\Psi_I^{(k)}$ will remain inside $|u| \leq c k |I|$.  It is also clear from \eqref{tiop} that $||\Psi_I^{(k)}||_1 \leq || u^k \Psi_I||_1$.  With only slightly more effort (using Minkowski's inequality), one can conclude that
\[ || \Psi_I^{(k)} ||_p \leq || u^k \Psi_I||_p \int_{[0,1]^k} \frac{d \theta_1 \cdots d \theta_k}{(\theta_1 + \cdots + \theta_k)^{\frac{1}{p'}}}. \]
The integral over the thetas will be finite for every $p < \infty$ (and, in fact, will be finite for $p = \infty$ as long as $k \neq 1$).  We can easily see that $||u^k \Psi_I||_p \approx \Lambda_I^{-k + \frac{1}{p'}}$ so by Young's inequality for convolutions we will have
\[ || \ind_{I} T_I F ||_p \lesssim \Lambda_I^{-k + \frac{1}{q} - \frac{1}{p}} || \ind_{I^*} F^{(k)} ||_q \]
when $\frac{1}{p} \leq \frac{1}{q} \leq 1$.  

Up to this point, the proofs of Theorems \ref{flashythm2} and \ref{flashythm} have coincided since \eqref{wanted2} and \eqref{wanted} share a common $||(\phi' \ell)^{-\gamma} f||_{L^p(\Omega)}$ on the right-hand side. To finish the proofs, we specialize at this point.  The proof of Theorem \ref{flashythm} proceeds by
estimating $\tilde S_3$ under the assumption $k - \gamma + \frac{1}{p} - \frac{1}{q} \geq 0$. We can multiply by $\Lambda_I^{\gamma}$ and exploit the fact that $\Lambda_I |I| \gtrsim 1$:
\begin{align*}
 \Lambda_I^{\gamma}  || \ind_I T_I F||_p &  \lesssim  \Lambda_I^{- ( k - \gamma + \frac{1}{p} - \frac{1}{q})} ||\ind_{I^*} F^{(k)}||_q \lesssim |I|^{k - \gamma + \frac{1}{p} - \frac{1}{q}} || \ind_{I^*} F^{(k)} ||_q. 
 \end{align*}
Now we raise this inequality to the $p$ power and sum over $I$.  Since $q \leq p$, the $\ell^p$ norm of these terms over $I$ will be dominated by the $\ell^q$ norm, so we are able to conclude exactly \eqref{wanted}:
\begin{align*}
 ||\tilde S_3||_p & \lesssim \left( \sum_{I \in {\mathcal G}}  \left( |I|^{ k - \gamma + \frac{1}{p} - \frac{1}{q} } || \ind_{I^*} F^{(k)}||_q \right)^p \right)^{\frac{1}{p}} \\ & \lesssim \left( \sum_{I \in {\mathcal G}}  \left( |I|^{k - \gamma + \frac{1}{p} - \frac{1}{q} } || \ind_{I^*} F^{(k)}||_q \right)^q \right)^{\frac{1}{q}} \\
  & \lesssim || \ell^{ k - \gamma + \frac{1}{p} - \frac{1}{q}} F^{(k)}||_{L^q(\Omega)}.
  \end{align*}

The case of Theorem \ref{flashythm} and \eqref{wanted2} is not much more work.  In this case, we take $k$ to be the largest integer not exceeding $\kappa + \alpha$.  We return to \eqref{tiop} and choose $N = k+1$ (when $k=0$, we can simply return to the definition \eqref{tidef}).  We need to estimate an expression of the form
\begin{equation} \int_{-c |I|}^{c |I|}  \frac{\Lambda_I |u|^k \ind_{|u| \leq c |I|}}{(1 + \Lambda_I |u|)^m} |F^{(k)}(x) - F^{(k)}(x + \theta u)| du \label{expression} \end{equation}
where $\theta$ is some positive real number between $0$ and $k+1$.
The shortest route to the desired estimate is to integrate by parts, integrating the difference of $F^{(k)}$ and differentiating the rest.  We conclude that
\begin{align*}
\int_{-c |I|}^{c |I|}   \frac{\Lambda_I |u|^k \ind_{|u| \leq c |I|}}{(1 + \Lambda_I |u|)^m} & |F^{(k)}(x) - F^{(k)}(x + \theta u)| du \lesssim_{m,k} \\
 & \frac{\Lambda_I |I|^k}{(1 + \Lambda_I |I|)^m} \int_0^{c|I|} |F^{(k)}(x) - F^{(k)}(x + \theta u)| du \\
 + & \frac{\Lambda_I |I|^k}{(1 + \Lambda_I |I|)^m} \int_{-c|I|}^0 |F^{(k)}(x) - F^{(k)}(x + \theta u)| du \\
 + & \int_{0}^{c|I|} \frac{\Lambda_I |s|^{k-1}}{(1 + \Lambda_I |s|)^m} \int_0^s |F^{(k)}(x) - F^{(k)}(x + \theta u)| du ds\\
  + & \int_{-c|I|}^0 \frac{\Lambda_I |s|^{k-1}}{(1 + \Lambda_I |s|)^m} \int_s^0 |F^{(k)}(x) - F^{(k)}(x + \theta u)| du ds.
\end{align*}
If we define
\begin{equation} M_{\beta} F^{(k)}(x) := \sup_{0 < \delta <  \ell(x)} \frac{1}{\delta^{1+\beta}} \int_{-\frac{\delta}{2}}^{\frac{\delta}{2}} |F^{(k)}(x + u) - F^{(k)}(x)| du, \label{smoothHLM} \end{equation}
we find that we are able to bound the expression \eqref{expression} by an implicit constant (depending on $m$, with $k$ considered fixed) times
\[ \frac{\Lambda_I |I|^k}{(1 + \Lambda_I |I|)^m} |I|^{1 + \beta} \theta^{\beta} M_\beta F^{(k)}(x) + \! \int_0^{\infty} \! \frac{\Lambda_I s^{k+\beta} \theta^{\beta} M_\beta F^{(k)}(x)}{(1 + \Lambda_I s)^m} ds \lesssim \frac{ \theta^{\beta} M_\beta F^{(k)}(x)}{\Lambda_I^{k + \beta}} \]
(once we take $m > k + \beta + 1$).
We choose $\beta$ so that $k + \beta = \gamma$.  If $k = 0$, then we apply this estimate to \eqref{tidef} with $\theta = 1$ (since $N=1$ in this case). To finish, multiply by $\Lambda_I^{\gamma}$ and take an $\ell^p$ norm over all intervals $I \in {\mathcal G}_L$ to conclude \eqref{wanted2}.  When $k \neq 0$, we use the same procedure, but we first write
\begin{align*}
\left| \vphantom{\sum_j^j} F^{(k)} \right. & \left. \left(  x + u \sum_{j=1}^k  \theta_j \right) -F^{(k)} \left(x + u+ u\sum_{j=1}^k \theta_j \right) \right| \leq \\
& \left|F^{(k)} \left(x + u \sum_{j=1}^k \theta_j \right) -F^{(k)} (x) \right| + \left|F^{(k)} (x) -F^{(k)} \left(x + u+ u\sum_{j=1}^k \theta_j \right) \right|
\end{align*}
We use Minkowski's inequality for each fixed $(\theta_1,\ldots,\theta_k) \in [0,1]^k$ and estimate each of these two differences exactly as was done when $k=0$.

\section{Appendix on the spaces $X^p_\beta(\Omega)$.}
\label{s.appendix}
We conclude with some further comments about the smoothness spaces in Theorems \ref{t.polynomial} and \ref{flashythm2}.  First, we establish the inequality
\begin{equation} || f||_{X^p_\beta(\Omega)} \lesssim_{p,\beta} ||D^{\beta} f||_{L^p(\R)}. \label{compineq} \end{equation}
When $\beta$ is an integer, $||f||_{X_\beta^p(\Omega)} = ||f^{(\beta)}||_{L^p(\Omega)} \lesssim_{p,\beta} || D^{\beta} f||_{L^p(\R)}$ simply by the observation that the $L^p$-norm of any function is comparable to the $L^p$-norm of its Hilbert transform for $p \in (1,\infty)$.

  If $\beta$ is not an integer, then $\beta = k + \alpha$ for some integer $k$ and $\alpha \in (0,1)$.  We have that $f^{(k)} = I_{\alpha} D^{\beta} f$, where $I_{\alpha}$ is convolution with a kernel $K_{-1+\alpha}$ which is homogeneous of degree $-1+\alpha$.  Taking a pointwise difference, we see that
\begin{align*}
 \left| f^{(k)}(x+u) - f^{(k)}(x)\right|  & = \left| I_{\alpha} D^{\beta} f(x+u) - I_{\alpha} D^{\beta} f(x) \right| \\
  & \leq \int \left|  K_{-1+\alpha} (x + u - y) - K_{-1 + \alpha}(x-y)) \right| \left| D^{\beta} f(y) \right| dy.
 \end{align*}
We employ the difference estimate
\begin{align*}
\left|  K_{-1+\alpha} (x + u) - K_{-1 + \alpha}(x)) \right| \lesssim_{\alpha} |x+u|^{-1+\alpha}  \ind_{|x+u| \leq |u|} & + |x|^{-1 + \alpha} \ind_{|x| \leq |u|} \\  + |u| \left| x + \frac{u}{2} \right|^{-2 + \alpha}& \ind_{\left| x + \frac{u}{2} \right| \geq |u|} 
\end{align*}
to conclude that
\begin{align*}
 \int & \left|  K_{-1+\alpha} (x + u - y) - K_{-1 + \alpha}(x-y)) \right|  \left| D^{\beta} f(y) \right| dy \\ & \lesssim_{\alpha,\beta} |u|^{\alpha} \left[ M D^{\beta} f(x+u) + M D^{\beta} f \left(x + \frac{u}{2} \right) + M D^{\beta} f(x) \right], 
 \end{align*}
where $M$ is still the Hardy-Littlewood maximal operator.  Consequently, when we apply $M_{\alpha}$ as defined in \eqref{smoothHLM} to $f^{(k)}$, we find
\[ | M_{\alpha} f^{(k)} (x)| \lesssim_{\alpha,\beta} M^2 D^{\beta} f(x). \]
Since $|| f||_p \approx_p || M f ||_p \approx_p ||M^2 f||_p$ for $p \in (1,\infty)$, we recover \eqref{compineq} in the non-integer case as well.  We note that the corresponding comparisons for the spaces ${{\mathscr L}^p_\beta}$ and ${\Lambda_\beta^{p,p}}$ mentioned in the introduction proceed similarly.

Finally, we give two representative examples of inequalities for the operators $M_{\alpha}$ which yield various embedding relationships among the spaces $X^p_{\beta}(\Omega)$ and weighted analogues. 

\begin{proposition}
Suppose $0 \leq \alpha \leq \alpha' \leq 1$.  Then for any measurable function $f$ on $\Omega$, we have
\[ M_{\alpha} f(x) \leq \ell(x)^{\alpha'-\alpha} M_{\alpha'} f(x). \]
\end{proposition}
\begin{proof} The proof is a trivial consequence of the fact that the supremum over $\delta$ only includes $\delta < \ell(x)$ at each $x$.
\end{proof}

 \begin{proposition}
 Suppose $p$ and $q$ are finite indices such that $1 \leq q < p$; let $\alpha \in (0,1)$ satisfy $\alpha + \frac{1}{p} = \frac{1}{q}$.  Then
 \[ || f ||_{p} \lesssim_{p,q} || \ell^{-\alpha} f ||_q + || M_{\alpha} f ||_q \]
 uniformly for all $f \in L^p(\Omega)$.
 \end{proposition}
 \begin{proof}
 This proof is essentially the same as the recent proof of an abstract Sobolev embedding theorem proved by the second author \cite{gressman2013II}.  Fix any $x$ and consider the functions
 \[ g(\delta) := \frac{1}{\delta} \int_{-\frac{\delta}{2}}^{\frac{\delta}{2}} |f(x+u)| du \mbox{ and } h(\delta) := \frac{1}{\delta} \int_{-\frac{\delta}{2}}^{\frac{\delta}{2}} |f(x+u)-f(x)| du \]
 If $g(\ell(x)) \geq \frac{1}{2} |f(x)|$, then trivially
 \[ |f(x)| \leq 2^{1-\theta} \left( \frac{1}{\ell(x)} \int_{-\frac{\ell(x)}{2}}^{\frac{\ell(x)}{2}} |f(x+u)| du \right)^{1-\theta} |f(x)|^{\theta} \]
 for any $\theta \in [0,1]$.
 Otherwise, $g(\ell(x))  < \frac{1}{2} |f(x)|$; since
 \[ |f(x)| \leq g(\delta) + h(\delta) \]
 we have $h(\ell(x)) > \frac{1}{2} |f(x)|$.  Because $h$ is continuous and tends to zero for almost every $x$ (i.e., the Lebesgue points of the function $f$) as $\delta \rightarrow 0^{+}$, there must be a $\delta$ at which $h(\delta) = \frac{1}{2} |f(x)|$, in which case $g(\delta) \geq \frac{1}{2} |f(x)|$ and so for this particular $\delta$ we have
 \[ |f(x)| \leq 2 \left( \frac{1}{\delta} \int_{-\frac{\delta}{2}}^{\frac{\delta}{2}} |f(x+u)| du \right)^{1-\theta} \left( \frac{1}{\delta} \int_{-\frac{\delta}{2}}^{\frac{\delta}{2}} |f(x+u)- f(x)| du \right)^{\theta}. \]
 Now because $f$ is known to belong {\it a priori} to $L^p$, we have
 \[ \frac{1}{\delta} \int_{-\frac{\delta}{2}}^{\frac{\delta}{2}} |f(x+u)| du \leq \delta^{-\frac{1}{p}} ||f||_p. \]
 Thus for almost every $x$ we have
 \begin{align*}
  |f(x)| \leq  2  ||f||_p^{1-\theta} \left[ \vphantom{ \int_{-\frac{\delta}{2}}^{\frac{\delta}{2}}} \right. &  (\ell(x))^{-\frac{1-\theta}{p}} |f(x)|^{\theta}  
  \\ & \left. + \sup_{\delta < \ell(x)} \left( \frac{1}{\delta^{1+ \frac{1-\theta}{p \theta}}} \int_{-\frac{\delta}{2}}^{\frac{\delta}{2}} |f(x+u)- f(x)| du \right)^{\theta}  \right]. 
  \end{align*}
 Now set $\theta := \frac{1}{1 + \alpha p}$ Notice $\frac{1-\theta}{p \theta} = \alpha$ and $\theta p = q$. Taking $L^p$-norms both sides gives
 \[ ||f||_p \leq 2 ||f||_{p}^{1-\theta} \left( || \ell^{-\alpha} f ||_q^{\theta} + || M_{\alpha} f ||_q^{\theta} \right). \]
 Since $||f||_p < \infty$ is assumed, the proposition is finished after dividing by $||f||_{p}^{1-\theta}$ and raising the inequality to the power $\theta^{-1}$.
 \end{proof}

\bibliography{mybib}

\def\cprime{$'$} \def\cydot{\leavevmode\raise.4ex\hbox{.}}
\begin{bibdiv}
\begin{biblist}

\bib{bc1984}{article}{
      author={Beals, R.},
      author={Coifman, R.~R.},
       title={Scattering and inverse scattering for first order systems},
        date={1984},
     journal={Comm. Pure Appl. Math.},
      volume={37},
      number={1},
       pages={39\ndash 90},
}

\bib{bdt1988}{book}{
      author={Beals, Richard},
      author={Deift, Percy},
      author={Tomei, Carlos},
       title={Direct and inverse scattering on the line},
      series={Mathematical Surveys and Monographs},
   publisher={American Mathematical Society},
     address={Providence, RI},
        date={1988},
      volume={28},
        ISBN={0-8218-1530-X},
      review={\MR{954382 (90a:58064)}},
}

\bib{cu2011}{article}{
      author={Chaudhury, Kunal~Narayan},
      author={Unser, Michael},
       title={On the {H}ilbert transform of wavelets},
        date={2011},
        ISSN={1053-587X},
     journal={IEEE Trans. Signal Process.},
      volume={59},
      number={4},
       pages={1890\ndash 1894},
         url={http://dx.doi.org/10.1109/TSP.2010.2103072},
      review={\MR{2807755 (2012a:42073)}},
}

\bib{dkmvz1999}{article}{
      author={Deift, P.},
      author={Kriecherbauer, T.},
      author={McLaughlin, K. T.-R.},
      author={Venakides, S.},
      author={Zhou, X.},
       title={Uniform asymptotics for polynomials orthogonal with respect to
  varying exponential weights and applications to universality questions in
  random matrix theory},
        date={1999},
        ISSN={0010-3640},
     journal={Comm. Pure Appl. Math.},
      volume={52},
      number={11},
       pages={1335\ndash 1425},
  url={http://dx.doi.org/10.1002/(SICI)1097-0312(199911)52:11<1335::AID-CPA1>3.0.CO;2-1},
      review={\MR{1702716 (2001g:42050)}},
}

\bib{dvz1997}{article}{
      author={Deift, P.},
      author={Venakides, S.},
      author={Zhou, X.},
       title={New results in small dispersion {K}d{V} by an extension of the
  steepest descent method for {R}iemann-{H}ilbert problems},
        date={1997},
     journal={Internat. Math. Res. Notices},
      number={6},
       pages={286\ndash 299},
}

\bib{dz1993}{article}{
      author={Deift, P.},
      author={Zhou, X.},
       title={A steepest descent method for oscillatory {R}iemann-{H}ilbert
  problems. {A}symptotics for the {MK}d{V} equation},
        date={1993},
     journal={Ann. of Math. (2)},
      volume={137},
      number={2},
       pages={295\ndash 368},
}

\bib{dz2002}{article}{
      author={Deift, Percy},
      author={Zhou, Xin},
       title={Perturbation theory for infinite-dimensional integrable systems
  on the line. {A} case study},
        date={2002},
        ISSN={0001-5962},
     journal={Acta Math.},
      volume={188},
      number={2},
       pages={163\ndash 262},
         url={http://dx.doi.org/10.1007/BF02392683},
      review={\MR{1947893 (2005m:37172)}},
}

\bib{dz2003}{article}{
      author={Deift, Percy},
      author={Zhou, Xin},
       title={Long-time asymptotics for solutions of the {NLS} equation with
  initial data in a weighted {S}obolev space},
        date={2003},
        ISSN={0010-3640},
     journal={Comm. Pure Appl. Math.},
      volume={56},
      number={8},
       pages={1029\ndash 1077},
         url={http://dx.doi.org/10.1002/cpa.3034},
        note={Dedicated to the memory of J{\"u}rgen K. Moser},
      review={\MR{1989226 (2004k:35349)}},
}

\bib{do2011}{article}{
      author={Do, Yen},
       title={A nonlinear stationary phase method for oscillatory
  {R}iemann-{H}ilbert problems},
        date={2011},
     journal={Int. Math. Res. Not. IMRN},
      number={12},
       pages={2650\ndash 2765},
}

\bib{gressman2013II}{unpublished}{
      author={Gressman, P.~T.},
       title={{F}ractional {P}oincar\'{e} and logarithmic {S}obolev
  inequalities for measure spaces},
        note={arXiv:1205.5773, to appear in {\it J. Func. Anal.}},
}

\bib{its1981}{article}{
      author={It{\cydot}s, A.~R.},
       title={Asymptotic behavior of the solutions to the nonlinear
  {S}chr\"odinger equation, and isomonodromic deformations of systems of linear
  differential equations},
        date={1981},
        ISSN={0002-3264},
     journal={Dokl. Akad. Nauk SSSR},
      volume={261},
      number={1},
       pages={14\ndash 18},
      review={\MR{636848 (84k:58238)}},
}

\bib{ll2008}{article}{
      author={Levin, Eli},
      author={Lubinsky, Doron~S.},
       title={Universality limits in the bulk for varying measures},
        date={2008},
        ISSN={0001-8708},
     journal={Adv. Math.},
      volume={219},
      number={3},
       pages={743\ndash 779},
         url={http://dx.doi.org/10.1016/j.aim.2008.06.010},
      review={\MR{2442052 (2010a:60009)}},
}

\bib{lubinsky2009}{article}{
      author={Lubinsky, Doron~S.},
       title={A new approach to universality limits involving orthogonal
  polynomials},
        date={2009},
        ISSN={0003-486X},
     journal={Ann. of Math. (2)},
      volume={170},
      number={2},
       pages={915\ndash 939},
         url={http://dx.doi.org/10.4007/annals.2009.170.915},
      review={\MR{2552113 (2011a:42042)}},
}

\bib{mm2006}{article}{
      author={McLaughlin, K. T.-R.},
      author={Miller, P.~D.},
       title={The {$\overline{\partial}$} steepest descent method and the
  asymptotic behavior of polynomials orthogonal on the unit circle with fixed
  and exponentially varying nonanalytic weights},
        date={2006},
     journal={IMRP Int. Math. Res. Pap.},
       pages={Art. ID 48673, 1\ndash 77},
}

\bib{mm2008}{article}{
      author={McLaughlin, K. T.-R.},
      author={Miller, P.~D.},
       title={The {$\overline{\partial}$} steepest descent method for
  orthogonal polynomials on the real line with varying weights},
        date={2008},
        ISSN={1073-7928},
     journal={Int. Math. Res. Not. IMRN},
       pages={Art. ID rnn 075, 66},
}

\bib{ps1992}{article}{
      author={Phong, D.~H.},
      author={Stein, E.~M.},
       title={Oscillatory integrals with polynomial phases},
        date={1992},
     journal={Invent. Math.},
      volume={110},
      number={1},
       pages={39\ndash 62},
}

\bib{pss1999}{article}{
      author={Phong, D.~H.},
      author={Stein, E.~M.},
      author={Sturm, J.~A.},
       title={On the growth and stability of real-analytic functions},
        date={1999},
     journal={Amer. J. Math.},
      volume={121},
      number={3},
       pages={519\ndash 554},
}

\bib{steinsi}{book}{
      author={Stein, Elias~M.},
       title={Singular integrals and differentiability properties of
  functions},
      series={Princeton Mathematical Series, No. 30},
   publisher={Princeton University Press},
     address={Princeton, N.J.},
        date={1970},
}

\bib{varzugin1996}{article}{
      author={Varzugin, G.~G.},
       title={Asymptotics of oscillatory {R}iemann-{H}ilbert problems},
        date={1996},
     journal={J. Math. Phys.},
      volume={37},
      number={11},
       pages={5869\ndash 5892},
}

\bib{wright2011}{article}{
      author={Wright, James},
       title={From oscillatory integrals to complete exponential sums},
        date={2011},
     journal={Math. Res. Lett.},
      volume={18},
      number={2},
       pages={231\ndash 250},
}

\bib{zhou1998}{article}{
      author={Zhou, Xin},
       title={{$L^2$}-{S}obolev space bijectivity of the scattering and inverse
  scattering transforms},
        date={1998},
        ISSN={0010-3640},
     journal={Comm. Pure Appl. Math.},
      volume={51},
      number={7},
       pages={697\ndash 731},
  url={http://dx.doi.org/10.1002/(SICI)1097-0312(199807)51:7<697::AID-CPA1>3.0.CO;2-1},
      review={\MR{1617249 (2000c:34220)}},
}

\end{biblist}
\end{bibdiv}

\end{document}